\documentclass[11pt]{amsart}
\usepackage{amsmath,amssymb}
\usepackage{amsfonts}
\usepackage{amscd}
\usepackage{lscape}
\usepackage{latexsym}
\usepackage{amsfonts,amssymb,stmaryrd,amscd,amsmath,latexsym,amsbsy}

\newcommand{\nc}{\mathbb C}
\newcommand{\nz}{\mathbb Z}

\newcommand{\snz}{\mbox{\scriptsize{${\mathbb Z}$}}}

\newcommand{\gtg}{\mathfrak g}

\newcommand{\gtsl}{\mathfrak{sl}}
\newcommand{\gth}{\mathfrak h}
\newcommand{\gtt}{\mathfrak t}

\newcommand{\tf}{\widetilde{f}}
\newcommand{\te}{\widetilde{e}}
\newcommand{\eps}{\varepsilon}
\newcommand{\vphi}{\varphi}
\newcommand{\cB}{{\mathcal B}}
\newcommand{\cM}{{\mathcal M}}
\newcommand{\cBZ}{{\mathcal{BZ}}}

\newcommand{\onu}{\overline{\nu}}
\newcommand{\ophi}{\overline{\phi}}

\newcommand{\dimc}{\dim_{\mbox{\scriptsize{${\mathbb C}$}}}}
\newcommand{\homc}{\mbox{Hom}_{\mbox{\scriptsize{${\mathbb C}$}}}}

\newcommand{\dimv}{\underline{\dim}}

\newcommand{\htf}{\widehat{{f}}}
\newcommand{\hte}{\widehat{{e}}}
\newcommand{\heps}{\widehat{\varepsilon}}
\newcommand{\hvphi}{\widehat{\varphi}}

\newcommand{\out}{\mbox{out}}
\newcommand{\sout}{\mbox{\scriptsize{out}}}
\newcommand{\inn}{\mbox{in}}
\newcommand{\sinn}{\mbox{\scriptsize{in}}}

 \newtheorem{thm}{Theorem}[subsection]
 \newtheorem{prop}[thm]{Proposition}
 \newtheorem{lemma}[thm]{Lemma}
 \newtheorem{cor}[thm]{Corollary}
 \newtheorem{defn}[thm]{Definition}

 \newtheorem{rem}{Remark}

\newenvironment{Ac}%
 {\hspace*{-1.2em}\textbf{Acknowledgment.}\hspace{0.7em}}{}
 {\hspace*{-1.2em}\textbf{Notation.}\hspace{0.7em}}{}

\title[Toward Berenstein-Zelevinsky data in affine type $A$, part III]
{Toward Berenstein-Zelevinsky data in affine type $A$, \\
part III: Proof of the connectedness}
\author{Satoshi Naito, Daisuke Sagaki, and Yoshihisa Saito}
\dedicatory{
Dedicated to Professor Michio Jimbo on the occasion of his sixtieth
birthday.}

\address{Satoshi Naito: Department of Mathematics, Tokyo Institute of 
Technology, 2-12-1, Oh-Okayama Meguro-ku, Tokyo 152-8551, Japan.}
\email{naito@math.titech.ac.jp}
\address{Daisuke Sagaki: Institute of Mathematics, University of Tsukuba, 
Ibaraki 305-8571, Japan.}
\email{sagaki@math.tsukuba.ac.jp}
\address{Yoshihisa Saito: Graduate School of Mathematical Sciences, 
University of Tokyo, 3-8-1, Komaba
Meguro-ku, Tokyo 153-8914, Japan.}
\email{yosihisa@ms.u-tokyo.ac.jp}


\begin{document}
\bigskip
\begin{abstract}
We prove the connectedness of the crystal
$\left(\cBZ_{\snz}^{\sigma};\mbox{wt},\widehat{\eps}_p,
\widehat{\vphi}_p,\hte_p,\htf_p\right)$, which we introduced in \cite{NSS}. 
\end{abstract}
\maketitle
\section{Introduction}
This paper is a continuation of our previous works (\cite{NSS} and \cite{NSS2}).
In \cite{NSS}, motivated by the works (\cite{Kam1} and \cite{Kam2}) of Kamnitzer on 
Mirkovi\'c-Vilonen polytopes in finite types, we introduced an affine
analog of Berenstein-Zelevinsky datum (BZ datum for short) in  
type $A_{l-1}^{(1)}$. Let us recall its construction briefly. 
For a finite interval $I$ in $\nz$,  we denote by $\cBZ_I$ the set of those BZ data
of type $A_{|I|}$ which satisfy  a certain normalization condition, called the 
$w_0$-normalization condition in \cite{NSS2}. The family
$\{\cBZ_I~|~\mbox{$I$ is a finite interval in $\nz$}\}$ forms a projective
system, and hence the set $\cBZ_{\nz}$ of BZ data of type $A_{\infty}$ is defined
to be a kind of projective limit of this projective system.   
Furthermore, for $l\geq 3$, we define the set $\cBZ_{\nz}^{\sigma}$ of BZ data 
of type $A_{l-1}^{(1)}$ to be the fixed point subset of $\cBZ_{\nz}$ under a natural
action of the automorphism $\sigma:\nz\to\nz$ given by $\sigma(j)=j+l$ for 
$j\in\nz$.
Note that a BZ datum of type 
$A_{l-1}^{(1)}$ is realized as a collection of those integers, indexed by the set 
of infinite Maya diagrams, which satisfy the ``{\it edge inequalities}'', 
``{\it tropical Pl\"ucker relations}'', 
and some additional conditions (see Definition \ref{defn:BZ} and
\ref{defn:affineBZ} for details). 
The set $\cBZ_{\nz}^{\sigma}$ has a $U_q(\widehat{\gtsl}_l)$-crystal
structure, which is naturally induced by that on
$\cBZ_I$. 
In \cite{NSS}, we proved that there exists a distinguished connected component 
$\cBZ_{\nz}^{\sigma}({\bf O})$ of $\cBZ_{\nz}^{\sigma}$, which is isomorphic 
as a crystal to the crystal basis $B(\infty)$ of the negative part 
$U_q^-(\widehat{\gtsl}_{l})$ of $U_q(\widehat{\gtsl}_{l})$. 
We anticipated that the connected component 
$\cBZ_{\nz}^{\sigma}({\bf O})$ is identical to the whole of
$\cBZ_{\nz}^{\sigma}$, but we could not prove it in \cite{NSS}.
The purpose of this paper is prove the anticipated identity, that is, to prove 
the connectedness of the crystal graph of $\cBZ_{\nz}^{\sigma}$. 

In \cite{NSS2}, we introduced the notion of $e$-BZ data of type $A_{l-1}^{(1)}$,
which are defined in the same way as BZ data with another 
normalization condition, called the $e$-normalization condition in 
\cite{NSS2}. In this paper, we mainly treat $e$-BZ data
instead of BZ data for the following reasons. First, it is known that 
the set $(\cBZ_{\nz}^e)^{\sigma}$ of $e$-BZ data of type $A_{l-1}^{(1)}$ is
isomorphic as a crystal to $\cBZ_{\nz}^{\sigma}$, and hence the connectedness 
of the crystal graph of $\cBZ_{\nz}^{\sigma}$ is equivalent to that of 
$(\cBZ_{\nz}^e)^{\sigma}$.
Second, in \cite{NSS2}, we showed that there is a natural correspondence 
between $e$-BZ data and (certain) limits of irreducible
Lagrangians of varieties associated to quivers of finite type $A$. 
Thus, we can use geometrical (or quiver-theoretical) methods for 
the study of $e$-BZ data. This is an advantage of $e$-BZ data.    

Our main result (Theorem \ref{thm:main}) states that the crystal
$(\cBZ_{\nz}^e)^{\sigma}$ is isomorphic to $B(\infty)$. 
Because we already know that a distinguished connected component of 
$(\cBZ_{\nz}^e)^{\sigma}$ is isomorphic to $B(\infty)$, 
Theorem \ref{thm:main} tells us that this connected component 
is identical to the whole of $(\cBZ_{\nz}^e)^{\sigma}$.
In other words, we obtain a new explicit realization
of $B(\infty)$ in terms of an affine analog of a BZ datum. 
Our strategy for proving Theorem \ref{thm:main} is as follows.
In \cite{KS}, Kashiwara and the third author gave those conditions which
characterize $B(\infty)$ uniquely (see Theorem \ref{thm:seven-cond}
for details). We will establish Theorem \ref{thm:main} by verifying that the 
$(\cBZ_{\nz}^e)^{\sigma}$ indeed satisfies these conditions. 

This paper is organized as follows. In Section 2, we give a quick review of
results in our previous works. In Section 3, we introduce a new 
crystal structure,  called the {\it ordinary crystal structure}, on
$\cBZ_I^e$. Here, $\cBZ_I^e$ is the set of $e$-BZ data associated to a finite 
interval $I$. Since $(\cBZ_{\nz}^e)^{\sigma}$ is the set of $\sigma$-fixed
points of a kind of projective limit of $\cBZ_I^e$'s, we can define
the ordinary crystal structure on $(\cBZ_{\nz}^e)^{\sigma}$ induced naturally
by that of $\cBZ_I^e$'s. However, in order to overcome some technical
difficulties in following this procedure, we need a quiver-theoretical
interpretation of $\cBZ_I^e$. We treat these technicalities in Section 4.
In Section 5, we prove our main result (Theorem \ref{thm:main}) 
by checking the conditions in Theorem \ref{thm:seven-cond} for the 
$(\cBZ_{\nz}^e)^{\sigma}$. In our proof, the ordinary
crystal structure on $(\cBZ_{\nz}^e)^{\sigma}$ introduced in
Section 4, plays a crucial role.

Finally, let us mention some related works, which appeared recently.
The first one is by Muthiah \cite{M}. In \cite{BFG}, Braverman, Finkelberg, and
Gaitsgory introduced analogs of Mirkovi\'c-Vilonen cycles in the case of an 
affine Kac-Moody group, and defined a crystal structure on the set of 
those cycles. After that, Muthiah studied the crystal structure of those cycles in an 
explicit way, and proved that it is isomorphic to the crystal $\cBZ_{\nz}^{\sigma}$ 
in affine type $A$. 
The second one is by Baumann, Kamnitzer, and Tingley \cite{BKT}.
Let $\gtg$ be a symmetric affine Kac-Moody Lie algebra.
In \cite{BKT}, they introduced the notion of affine Mirkovi\'c-Vilonen polytopes 
by using the theory of preprojective algebras of the same type as $\gtg$,
and showed that there exists a bijection 
between the set of affine Mirkovi\'c-Vilonen polytopes and the crystal
basis $B(-\infty)$ of the positive part $U_q^+(\gtg)$ of $U_q(\gtg)$.
It seems to us that these works are closely related to results in this
paper. However, an explicit relationship between them is still 
unclear; this is our future problem. 

\medskip
\begin{Ac}
Research of SN is supported in part by Grant-in-Aid for Scientific 
Research (C), No.20540006. 
Research of DS is supported in part by Grant-in-Aid for Young 
Scientists (B), No.19740004. 
Research of YS is supported in part by Grant-in-Aid for Scientific 
Research (C), No.20540009.
\end{Ac}
  
\section{Review of known results}
\subsection{Preliminaries on root data}
Let $\gtt$ be a vector space over $\nc$ with basis 
$\{\epsilon_i\}_{i\in \snz}$; we set
$h_i:=\epsilon_i-\epsilon_{i+1}$, $i\in\nz$.
We define $\Lambda_i,~\Lambda_i^c\in \gtt^*:=\homc(\gtt,\nc)$, $i\in \nz$ by
$$
\langle\epsilon_j,\Lambda_i\rangle_{\snz}:=\left\{\begin{array}{rl}
1 & \mbox{if }j\leq i,\\
0 & \mbox{if }j>i,
\end{array}\right.\qquad
\langle\epsilon_j,\Lambda_i^c\rangle_{\snz}=\left\{\begin{array}{rl}
0 & \mbox{if }j\leq i,\\
1 & \mbox{if }j>i,
\end{array}\right.
$$
where $\langle\cdot,\cdot\rangle_{\snz}:\gtt\times\gtt^*\to\nc$ 
is the canonical pairing, and set
$\alpha_i:=-\Lambda_{i-1}+2\Lambda_i-\Lambda_{i+1}$, $i\in\nz.$
Let $W_{\nz}:=\langle \sigma_i~|~i\in I
\rangle(\subset GL(\gtt))$ be the Weyl group of type $A_{\infty}$, where
$\sigma_i$ is the automorphism of $\gtt$ defined by
$\sigma_i(t)=t-\langle t,\alpha_i\rangle_{\nz}h_i$, $t\in\gtt$;  
the group $W_{\nz}$ also acts on $\gtt^*$ by 
$\sigma_i(\lambda)=\lambda-\langle h_i,\lambda\rangle_{\nz}\alpha_i$, 
$\lambda\in\gtt^*$.

Let $I=[n+1,n+m]$ be a finite interval in $\nz$ whose cardinality is equal 
to $m$, and consider a finite-dimensional subspace 
$\gth_I:=\bigoplus_{i\in I}\nc h_i$ of $\gtt^*$. 
For each $i\in I$, set $\alpha_i^I:=\pi_I(\alpha_i)$ and $\varpi_i^I
:=\pi_I(\Lambda_i)$,
where $\pi_I:\gtt^*\to\gth_I^*:=\homc(\gth_I,\nc)$ is the natural projection;
we denote by $\langle\cdot,\cdot\rangle_I$ the canonical pairing between
$\gth_I$ and $\gth_I^*$. Then we can regard 
$(\{\alpha_i^I\}_{i\in I},\{h_i\}_{i\in I},\gth_I^*,\gth_I)$ as the root datum of type
$A_m$. Also, the set $\{\varpi_i^I\}_{i\in I}$ can be regarded as the set of
fundamental weights.  Let $W_I$ be the subgroup of $W_{\nz}$ generated by 
$\sigma_i$, $i\in I$. Since $\sigma_i$ stabilizes the subspace $\gth_I$
of $\gtt$ for all $i\in I$, we can regard $W_I$ as a subgroup of $GL(\gth_I)$; the
group $W_I$ acts on $\gth_I^*$ in a usual way. 
\subsection{BZ data associated to a finite interval}
Set $\widetilde{I}:=I\cup \{n+m+1\}$. A subset ${\bf k}\subset \widetilde{I}$ is 
called a Maya diagram associated to $I$; we denote by $\cM_I$ the set
of all Maya diagrams associated to $I$, and set $\cM_I^{\times}:=\cM_I\setminus
\{\phi,\widetilde{I}\}$. We identify $\cM_I^{\times}$ with $\Gamma_I:=
\bigsqcup_{i\in I}W_I\varpi_i^I$ via the bijection $[n+1,i]\leftrightarrow \varpi_i^I$. 
Under this identification, $\langle\cdot,\cdot\rangle_I$ induces a 
pairing between $\gth_I$ and $\cM_I^{\times}$, which is given explicitly 
as follows:
$$\langle h_i,{\bf k}\rangle_I=\begin{cases}
1 & \mbox{if $i\in {\bf k}$ and $i+1\not\in{\bf k}$},\\
-1 & \mbox{if $i\not\in {\bf k}$ and $i+1\in{\bf k}$},\\
0 & \mbox{otherwise}.
\end{cases}\eqno{(2.2.1)}$$

Let ${\bf M}=(M_{\bf k})_{{\bf k}\in\cM_I^{\times}}$ be a collection of integers
indexed by $\cM_I^{\times}$. For each ${\bf k}\in \cM_I^{\times}$, we call
$M_{\bf k}$ the ${\bf k}$-component of ${\bf M}$, and denote it by 
$({\bf M})_{\bf k}$.

\begin{defn}\label{defn:fin-BZ}
{\rm (1)} A collection ${\bf M}=(M_{\bf k})_{{\bf k}\in\cM_I^{\times}}$ of integers
indexed by $\cM_I^{\times}$ is called a Berenstein-Zelevinsky datum {\rm (}BZ
datum for short{\rm )} associated to $I$ if it satisfies the following 
conditions:
\vskip 1mm
\noindent
{\rm (BZ-1)} for all indices $i\ne j$ in $\widetilde{I}$ and all 
${\bf k}\in \cM_I$ such that ${\bf k}\cap\{i,j\}=\phi$,
$$M_{{\bf k}\cup\{i\}}+M_{{\bf k}\cup\{j\}}\leq 
M_{{\bf k}\cup\{i,j\}}+M_{\bf k};$$
{\rm (BZ-2)} for all indices $i<j<k$ in $\widetilde{I}$ and all
${\bf k}\in \cM_I$ such that ${\bf k}\cap\{i,j,k\}=\phi$,
$$M_{{\bf k}\cup\{i,k\}}+M_{{\bf k}\cup\{j\}}=\mbox{\rm min}
\left\{M_{{\bf k}\cup\{i,j\}}+M_{{\bf k}\cup\{k\}},~
M_{{\bf k}\cup\{j,k\}}+M_{{\bf k}\cup\{i\}}\right\}.$$
Here, $M_{\phi}=M_{\widetilde{I}}=0$ by convention.
\vskip 2mm
\noindent
{\rm (2)} A BZ datum ${\bf M}=(M_{\bf k})_{{\bf k}\in\cM_I^{\times}}$ is called
a $w_0$-BZ {\rm (}{\it resp}., $e$-BZ{\rm )} datum if it satisfies the
following normalization condition:
\vskip 1mm
\noindent
{\rm (BZ-0)} for every $i\in I$,
$M_{[i+1,n+m+1]}=0$ {\rm (}{\it resp.}, $M_{[n+1,i]}=0${\rm )}.
\vskip 1mm
\noindent
We denote by $\cBZ_I$ {\rm (}{\it resp}., $\cBZ_I^e${\rm )} the set of all
$w_0$-BZ {\rm (}{\it resp}., $e$-BZ{\rm )} data.
\end{defn}

For ${\bf M}=(M_{\bf k})_{{\bf k}\in\cM_I^{\times}}\in\cBZ_I$, define a new 
collection ${\bf M}^*=(M_{\bf k}^*)_{{\bf k}\in\cM_I^{\times}}$ of integers by
$$M_{\bf k}^*:=M_{{\bf k}^c},$$
where ${\bf k}^c:=\widetilde{I}\setminus {\bf k}$ is the complement of 
${\bf k}$ in $\widetilde{I}$. Then, ${\bf M}^*$ is an element of $\cBZ_I^e$,
and the map $\ast:{\bf M}\mapsto {\bf M}^*$ gives a bijection from
$\cBZ_I$ to $\cBZ_I^e$. We also denote its inverse by $\ast$. \\

Let $K=[n'+1,n'+m']$ be a subinterval of $I$, and define 
$$\cM_I^{\times}(K):=\{{\bf k}\in\cM_I^{\times}~|~{\bf k}=[n+1,n']\cup {\bf k}'
\mbox{ for some }{\bf k}'\in \cM_K^{\times}\}.$$
Then, $\cM_K^{\times}$ is naturally identified 
with $\cM_I^{\times}(K)$ via the bijection 
${\bf k}'\mapsto [n+1,n']\cup{\bf k}'$. We denote its inverse by
$\mbox{res}_K^I:\cM_I^{\times}(K)\overset{\sim}{\to}\cM_K^{\times}$.

For ${\bf M}=(M_{\bf k})_{{\bf k}\in\cM_I^{\times}}\in \cBZ_I^e$, we define
a new collection ${\bf M}_K=(M_{\bf m}')_{{\bf m}\in\cM_K^{\times}}$ of
integers indexed by $\cM_K^{\times}$ by
$$M_{\bf m}':=M_{(\mbox{res}_K^I)^{-1}({\bf m})}.$$
Then, ${\bf M}_K$ is an $e$-BZ datum associated to $K$.

\subsection{Crystal structure on BZ data associated to a finite interval}
First, we recall the crystal structure on $\cBZ_I$. 
For ${\bf M}=(M_{\bf k})_{{\bf k}\in\cM_I^{\times}}\in\cBZ_I$ and $i\in I$, 
define 
$$\mbox{wt}({\bf M}):=\sum_{i\in I}M_{[n+1,i]}\alpha_i^I,$$
$$\eps_i({\bf M}):=-\left(M_{[n+1,i]}+M_{[n+1,i-1]\cup\{i+1\}}-
M_{[n+1,i-1]}-M_{[n+1,i+1]}\right),$$
$$\vphi_i({\bf M}):=\eps_i({\bf M})+\langle h_i,
\mbox{wt}({\bf M})\rangle_I.$$

\begin{prop}\label{prop:ord-action}
{\rm (1)} Let ${\bf M}=(M_{\bf k})_{{\bf k}\in\cM_I^{\times}}\in\cBZ_I$. If 
$\eps_i({\bf M})>0$, then there exits a unique $w_0$-BZ datum ${\bf M}'=
(M_{\bf k}')_{{\bf k}\in\cM_I^{\times}}$ such that
\begin{itemize}
\item[(i)] $M'_{[n+1,i]}=M_{[n+1,i]}+1$,
\item[(ii)] $M'_{\bf k}=M_{\bf k}$ for all ${\bf k}\in\cM_I^{\times}
\setminus \cM_I^{\times}(i)$,
\end{itemize} 
where $\cM_I^{\times}(i):=\{{\bf k}\in\cM_I^{\times}~|~i\in {\bf k}\mbox{ and }
i+1\not\in{\bf k}\}$.
\vskip 1mm
\noindent
{\rm (2)} There exits a unique $w_0$-BZ datum ${\bf M}''=
(M_{\bf k}'')_{{\bf k}\in\cM_I^{\times}}$ such that
\begin{itemize}
\item[(iii)] $M''_{[n+1,i]}=M_{[n+1,i]}-1$,
\item[(iv)] $M''_{\bf k}=M_{\bf k}$ for all ${\bf k}\in\cM_I^{\times}
\setminus \cM_I^{\times}(i)$.
\end{itemize} 
\end{prop}
We set
$$\te_i{\bf M}:=\begin{cases}
{\bf M}' & \mbox{if }~\eps_i({\bf M})>0,\\
0 & \mbox{if }~\eps_i({\bf M})=0,
\end{cases}\quad\mbox{and}\quad
\tf_i{\bf M}:={\bf M}''.$$
\begin{prop}
The set $\cBZ_I$, equipped with the maps $\mbox{\rm wt}$, $\eps_i$, 
$\vphi_i$, $\te_i$, $\tf_i$, is a crystal, which is isomorphic to
$\left(B(\infty);\mbox{\rm wt},\eps_i,\vphi_i,\te_i,\tf_i\right)$.
\end{prop}

The explicit form of the action of the lowering Kashiwara operator $\tf_i$
on $\cBZ_I$ is given by the following:
\begin{prop}\label{prop:ord-AM}
For ${\bf M}=(M_{\bf k})_{{\bf k}\in \cM_I^{\times}}\in \cBZ_I$, we have
$$(\tf_i{\bf M})_{\bf k}=\left\{\begin{array}{ll}
\mbox{\em min}\left\{M_{\bf k},~M_{s_i{\bf k}}
+c_i({\bf M})\right\} &
\mbox{if }~{\bf k}\in \cM_I^{\times}(i),\\
M_{\bf k} & \mbox{\em otherwise},
\end{array}\right.\eqno{(2.3.1)}$$
where $c_i({\bf M})=\langle h_i,\mbox{\rm wt}({\bf M})\rangle_I+
\eps_i({\bf M})-1$.
\end{prop}
Through the bijection $\ast: \cBZ_I\overset{\sim}{\to}\cBZ_I^e$, we can define
the $\ast$-crystal structure on $\cBZ_I^e$. Namely, for 
${\bf M}=(M_{\bf k})_{{\bf k}\in\cM_I^{\times}}\in\cBZ_I^e$, we set
$$\mbox{wt}({\bf M}):=\mbox{wt}({\bf M}^*),\quad
\eps_i^*({\bf M}):=\eps_i({\bf M}^*),\quad
\vphi_i^*({\bf M}):=\vphi_i({\bf M}^*),$$
$$\te_i^*:=\ast\circ\te_i\circ\ast,\quad\mbox{and}\quad
\tf_i^*:=\ast\circ\tf_i\circ\ast.$$

It is easy to obtain the following corollaries.
\begin{cor}\label{prop:ast-action}
Let ${\bf M}=(M_{\bf k})_{{\bf k}\in\cM_I^{\times}}\in\cBZ_I^e$. 
\vskip 1mm
\noindent
{\rm (1)} If $\eps_i^*({\bf M})>0$, then $\te_i^*{\bf M}$ is a unique 
$e$-BZ datum such that
\begin{itemize}
\item[(i)] $(\te_i^*{\bf M})_{[i+1,n+m+1]}=M_{[i+1,n+m+1]}+1$,
\item[(ii)] $(\te_i^*{\bf M})_{\bf k}=M_{\bf k}$ for all ${\bf k}\in\cM_I^{\times}
\setminus \cM_I^{\times}(i)^*$,
\end{itemize} 
where $\cM_I^{\times}(i)^*:=\{{\bf k}\in\cM_I^{\times}~|~
i\not\in {\bf k}\mbox{ and }i+1\in{\bf k}\}$.
\vskip 1mm
\noindent
{\rm (2)} $\tf_i^*{\bf M}$ is a unique $e$-BZ datum such that
\begin{itemize}
\item[(iii)] $(\tf_i^*{\bf M})_{[i+1,n+m+1]}=M_{[i+1,n+m+1]}-1$,
\item[(iv)] $(\tf_i^*{\bf M})_{\bf k}=M_{\bf k}$ for all ${\bf k}\in\cM_I^{\times}
\setminus \cM_I^{\times}(i)^*$.
\end{itemize} 
\vskip 1mm
\noindent
{\rm (3)}
For ${\bf M}=(M_{\bf k})_{{\bf k}\in \cM_I^{\times}}\in \cBZ_I^e$, we have
$$(\tf_i^*{\bf M})_{\bf k}=\left\{\begin{array}{ll}
\mbox{\em min}\left\{M_{\bf k},~M_{s_i{\bf k}}
+c_i^*({\bf M})\right\} &
\mbox{if }~{\bf k}\in \cM_n^{\times}(i)^*,\\
M_{\bf k} & \mbox{otherwise},
\end{array}\right.\eqno{(2.3.2)}$$
where $c_i^*({\bf M}):=\langle h_i,\mbox{\rm wt}({\bf M})\rangle_I+
\eps_i^*({\bf M})-1$.
\end{cor}

\begin{cor}
The set $\cBZ_I^e$, equipped with the maps $\mbox{\rm wt}$, $\eps_i^*$, 
$\vphi_i^*$, $\te_i^*$, $\tf_i^*$, is a crystal, which is isomorphic to
$\left(B(\infty);\mbox{\rm wt},\eps_i^*,\vphi_i^*,\te_i^*,\tf_i^*\right)$.
\end{cor}
\subsection{Lusztig data vs. BZ data}
Let $\Delta^+_I=\{(i,j)~|~i,j\in \widetilde{I}\mbox{ with }i<j\}$, and set 
$$\cB_I:=\left\{{\bf a}=(a_{i,j})_{(i,j)\in\Delta^+_I}~|~
a_{i,j}\in\nz_{\geq 0}\mbox{ for any }(i,j)\in \Delta^+_I\right\},$$
which is just the set of all $m(m+1)/2$-tuples of nonnegative integers 
indexed by $\Delta^+_I$. Here, $m$ is the cardinality of the interval $I$. 
An element of $\cB_I$ is called a Lusztig datum associated to $I$.

We define two crystal structures on $\cB_I$ 
(see \cite{S} and \cite{NSS2} for details). For ${\bf a}\in \cB_I$, set
$$\mbox{wt}({\bf a}):=-\sum_{i\in I}r_i\alpha_i^I,\quad\mbox{where}\quad
r_i:=\sum_{k=n+1}^i\sum_{l=i+1}^{n+m+1}a_{k,l},\quad i\in I.$$
For $i\in I$, we set
$$A^{(i)}_k({\bf a}):=\sum_{s=n+1}^k(a_{s,i+1}-a_{s-1,i}),\quad 
n+1\leq k\leq i,$$
$$A^{\ast (i)}_l({\bf a}):=\sum_{t=l+1}^{n+m+1}(a_{i,t}-a_{i+1,t+1}),\quad 
i\leq l\leq n+m+1,$$
where $a_{n,i}=0$ and $a_{i+1,n+m+2}=0$ by convention, and define 
$$\eps_i({\bf a}):=
\mbox{max}\left\{A_{n+1}^{(i)}({\bf a}),\ldots,A_i^{(i)}({\bf a})\right\},
\quad\vphi_i({\bf a}): = \eps_i({\bf a})+\langle h_i,\mbox{wt}({\bf a})
\rangle,$$
$$\eps_i^*({\bf a}):=\mbox{max}\left\{A^{\ast (i)}_i({\bf a}),\ldots,
A^{\ast (i)}_{n+m}({\bf a})\right\},\quad
\varphi_i^*({\bf a}):=\eps_i^*({\bf a})+\langle h_i,\mbox{wt}({\bf a})
\rangle.$$
Also, set
$$k_e:=\mbox{min}\left\{n+1\leq k\leq i\left|~\eps_i({\bf a})
=A_k^{(i)}({\bf a})\right.\right\},$$
$$k_f:=\mbox{max}\left\{n+1\leq k\leq i\left|~\eps_i({\bf a})
=A_k^{(i)}({\bf a})\right.\right\},$$
$$l_e:=\mbox{max}\left\{i\leq l\leq n+m\left|~\eps_i^*({\bf a})
=A_l^{\ast (i)}({\bf a})\right.\right\},$$
$$l_f:=\mbox{min}\left\{i\leq l\leq n+m\left|~\eps_i^*({\bf a})
=A_l^{\ast (i)}({\bf a})\right.\right\}.$$
For a given ${\bf a}\in \cB_I$, we define four $m(m+1)/2$-tuples of integers 
${\bf a}^{(p)}=\left(a_{k,l}^{(p)}\right)$, $p=1,2,3,4$, by 
\begin{align*}
{a}^{(1)}_{k,l}&:=\left\{\begin{array}{ll}
a_{k_e,i}+1 & \mbox{if }k=k_e,~l=i,\\
a_{k_e,i+1}-1 & \mbox{if }k=k_e,~l=i+1,\\
a_{k,l} & \mbox{otherwise}.
\end{array}\right.\\
{a}_{k,l}^{(2)}&:=\left\{\begin{array}{ll}
a_{k_f,i}-1 & \mbox{if }k=k_f,~l=i,\\
a_{k_f,i+1}+1 & \mbox{if }k=k_f,~l=i+1,\\
a_{k,l} & \mbox{otherwise},
\end{array}\right.\\
{a}_{k,l}^{(3)}&:=\left\{\begin{array}{ll}
a_{i,l_e+1}-1 & \mbox{if }k=i,~l=l_e+1,\\
a_{i+1,l_e+1}+1 & \mbox{if }k=i+1,~l=l_e+1,\\
a_{k,l} & \mbox{otherwise}.
\end{array}\right.\\
{a}^{(4)}_{k,l}&:=\left\{\begin{array}{ll}
a_{i,l_f+1}+1 & \mbox{if }k=i,~l=l_f+1,\\
a_{i+1,l_f+1}-1 & \mbox{if }k=i+1,~l=l_f+1,\\
a_{k,l} & \mbox{otherwise}.
\end{array}\right.
\end{align*}
Now, we define Kashiwara operators on $\cB_I$ as follows:
$$\te_i{\bf a}:=\left\{\begin{array}{ll}
{\bf 0}&\mbox{if }\eps_i({\bf a})=0,\\
{\bf a}^{(1)}&\mbox{if }\eps_i({\bf a})>0,
\end{array}\right.\quad\mbox{and}\quad \tf_i{\bf a}:={\bf a}^{(2)},$$
$$\te_i^*{\bf a}:=\left\{\begin{array}{ll}
{\bf 0}&\mbox{if }\eps_i^*({\bf a})=0,\\
{\bf a}^{(3)}&\mbox{if }\eps_i^*({\bf a})>0,
\end{array}\right.\quad\mbox{and}\quad \tf_i^*{\bf a}:={\bf a}^{(4)}.$$

\begin{prop}[\cite{S}]
Each of $(\cB_I,\mbox{\rm wt},\eps_i,\vphi_i,\te_i,\tf_i)$ and  
$(\cB_I,\mbox{\rm wt},\eps_i^*,\vphi_i^*,\te_i^*,\tf_i^*)$ is 
a crystal, which is isomorphic to $B(\infty)$.
\end{prop} 

Following \cite{NSS2}, we call the first one the ordinary crystal structure on 
$\cB_I$; the second one is called the $\ast$-crystal structure on 
$\cB_I$.

\begin{defn}[\cite{BFZ}]
Let ${\bf k}=\{k_{n+1}<k_{n+2}<\cdots<k_{n+u}\}\in\cM_I^{\times}$. 
For such a ${\bf k}$,
a ${\bf k}$-tableau is an upper-triangular matrix 
$C=(c_{p,q})_{n+1\leq p\leq q\leq n+u}$, with integer entries, satisfying
the condition
$$c_{p,p}=k_p,\quad n+1\leq p\leq n+u,$$
and the usual monotonicity condition for semi-standard tableaux:
$$c_{p,q}\leq c_{p,q+1},\quad c_{p,q}<c_{p+1,q}.$$
\end{defn}
For ${\bf a}=(a_{i,j})\in\cB_I$, define a collection ${\bf M}({\bf a})=
(M_{\bf k}({\bf a}))_{{\bf k}\in \cM_I^{\times}}$ of integers by
$$M_{\bf k}({\bf a}):=-\sum_{j=n+1}^{n+u}\sum_{i=n+1}^{k_j-1}a_{i,k_j}+
\mbox{min}\left\{\left.\sum_{n+1\leq p<q\leq n+u}a_{c_{p,q},c_{p,q}+(q-p)}
~\right|\begin{array}{c}C=(c_{p,q})\mbox{ is }\\ 
\mbox{a ${\bf k}$-tableau}\end{array}\right\}.$$
The following lemma is verified easily by direct calculation.

\begin{lemma}\label{lemma:BZ-int} 
Let ${\bf k}=\{k_{n+1}<k_{n+2}<\cdots < k_{n+u}\}$ be a Maya diagram
associated to $I$. 
\vskip 1mm
\noindent
{\rm (1)} If there exists $s$ such that $k_l=l$ for all $n+1\leq l\leq s$, 
then we have
\begin{align*}
M_{\bf k}({\bf a})&=-\sum_{j=s+1}^{n+u}\sum_{i=n+1}^{k_j-1}a_{i,k_j}
+\mbox{\rm min}\left\{\left.
\sum_{q=s+1}^{n+u}\sum_{p=n+1}^{q-1}a_{c_{p,q},c_{p,q}+(q-p)}
~\right|\begin{array}{c}C=(c_{p,q})~\mbox{\rm is }\\ 
\mbox{\rm a ${\bf k}$-tableau}\end{array}\right\}.
\end{align*}
In particular, $M_{\bf k}({\bf a})$ depends only on $a_{i,j}$ with $j\geq s+1$.
\vskip 1mm
\noindent
{\rm (2)} If there exists $t$ such that $k_{l-m+u-1}=l$ for all 
$t+1\leq l\leq n+m+1$, then we have
\begin{align*}
M_{\bf k}({\bf a})&=-\sum_{j=n+1}^{t-m+u-1}\sum_{i=n+1}^{k_j-1}a_{i,k_j}
-\sum_{j=t-m+u}^{n+u}\sum_{i=n+1}^{t}a_{i,j+m-u+1}\\
&\qquad\qquad
+\mbox{\rm min}\left\{\left.
\sum_{p=n+1}^{t-m+u-1}\sum_{q=p+1}^{n+u}
a_{c_{p,q},c_{p,q}+(q-p)}
~\right|\begin{array}{c}C=(c_{p,q})~\mbox{\rm is }\\ 
\mbox{\rm a ${\bf k}$-tableau}\end{array}\right\}.
\end{align*}
In particular, $M_{\bf k}({\bf a})$ depends only on $a_{i,j}$ with $i\leq t$.
\end{lemma}

\begin{thm}[\cite{BFZ}, \cite{S}]\label{thm:finA}
Let $\Psi_I$ denote the map ${\bf a}\mapsto {\bf M}({\bf a})$.
For every ${\bf a}\in\cB_I$, $\Psi_I({\bf a})={\bf M}({\bf a})$ is an 
$e$-BZ datum. Moreover, $\Psi_I:\cB_I\to \cBZ^{e}_I$ is an isomorphism 
of crystals with respect to the $\ast$-crystal structures on $\cB_I$ and 
$\cBZ_I^e$.
\end{thm}
\subsection{BZ data arising from the Lagrangian construction of $B(\infty)$}
Let $({I},H)$ be the double quiver of type $A_{m}$. 
Here, the finite interval ${I}=[n+1,n+m]$ in $\nz$ is considered 
as the set of vertices, and $H$ as the set of arrows. 
Let $\out(\tau)$ ({\it resp}., $\inn(\tau)$) denote the 
outgoing ({\it resp}., incoming) vertex of $\tau\in H$. For a given $\tau\in H$,
we denote by $\overline{\tau}$ the same edge as $\tau$ with the reverse 
orientation. Then, the map
$\tau\mapsto \overline{\tau}$ defines an involution of $H$. An orientation
$\Omega$ is a subset of $H$ such that $\Omega\cap\overline{\Omega}=\phi$ 
and $\Omega\cup\overline{\Omega}=H$. 
Note that $(I,\Omega)$ is a Dynkin quiver of type $A_m$. 

Let $\nu=(\nu_i)_{i\in I}\in \nz_{\geq 0}^I$. In the following, we regard $\nu$
as an element of $Q_+:=\bigoplus_{i\in I}\nz_{\geq 0}\alpha_i^I$ via 
the map $\nu\mapsto\sum_{i\in I}\nu_i\alpha_i^I$.  
Let $V(\nu)=\bigoplus_{i\in I}V(\nu)_i$ be an $I$-graded complex 
vector space with dimension vector $\dimv V(\nu)=\nu$. Set
$$E_{V(\nu),\Omega}:=\bigoplus_{\tau\in\Omega}
\homc(V(\nu)_{\sout(\tau)}, V(\nu)_{\sinn(\tau)}),\quad 
X(\nu):=\mathop{\bigoplus}_{\tau\in H}\homc(V(\nu)_{\sout(\tau)},
V(\nu)_{\sinn(\tau)}).$$
We will write an element of $E_{V(\nu),\Omega}$ or $X(\nu)$ as 
$B=(B_{\tau})$, where 
$B_{\tau}$ is an element of 
$\homc(V(\nu)_{\sout(\tau)},V(\nu)_{\sinn(\tau)})$. Define a 
symplectic form $\omega$ on $X(\nu)$ by
$$\omega(B,B'):=
\sum_{\tau\in H}\eps(\tau)\mbox{tr}(B_{\overline{\tau}}B_{\tau}'),$$
where $\eps(\tau)=1$ for $\tau\in\Omega$ and $\eps(\tau)=-1$ for $\tau\in
\overline{\Omega}$, and regard $X(\nu)$ as the cotangent bundle 
$T^*E_{V(\nu),\Omega}$ of $E_{V(\nu),\Omega}$ via the symplectic form $\omega$.  

Also, the reductive group $G(\nu):=\prod_{i\in I}
GL(V(\nu)_i)$ acts on $E_{V(\nu),\Omega}$ and $X(\nu)$ by :
$(B_{\tau})\mapsto (g_{\sinn(\tau)}B_{\tau}g_{\sout(\tau)}^{-1})$
for $g=(g_i)\in G(\nu)$.
Since the action of $G(\nu)$ on $X(\nu)$ preserves the symplectic form 
$\omega$, we can consider the corresponding moment map
$\mu:X(\nu)\to \bigl(\gtg(\nu)\bigr)^*\cong \gtg(\nu)$. Here $\gtg(\nu)=
\mbox{Lie }G(\nu)$,
and we identify $\gtg(\nu)$ with its dual via the Killing form. We set
$$\Lambda(\nu):=\mu^{-1}(0).$$
Then, $\Lambda(\nu)$ is a $G(\nu)$-invariant closed Lagrangian subvariety of 
$X(\nu)$. We denote by $\mbox{Irr}\Lambda(\nu)$ the set of all irreducible 
components of $\Lambda(\nu)$.\\

Let $\nu,\nu',\overline{\nu}\in Q_+$, with $\nu=\nu'+\overline{\nu}$.
Consider the diagram
$$\Lambda(\nu')\times\Lambda(\overline{\nu})\overset{q_1}{\longleftarrow}
\Lambda(\nu',\onu)\overset{q_2}{\longrightarrow}\Lambda(\nu).
\eqno{(2.5.1)}$$
Here, $\Lambda(\nu',\onu)$ denotes the variety of those $(B,\phi',\ophi)$ 
for which $B\in \Lambda(\nu)$, and $\phi'=(\phi_i'),~\ophi=(\ophi_i)$ 
give an exact sequence of $I$-graded complex vector spaces
$$0\longrightarrow V(\nu')_i\overset{\phi_i'}{\longrightarrow} V(\nu)
\overset{\ophi_i}{\longrightarrow}V(\onu)\longrightarrow 0$$
such that $\mbox{Im }\phi'$ is stable by $B$; note that $B$ induces 
$B':V(\nu')\to V(\nu')$ and $\overline{B}:V(\onu)\to V(\onu)$. 
The maps $q_1$ and $q_2$ are defined by
$q_1(B,\phi',\ophi):=(B',\overline{B})$ and $q_2(B,\phi',\ophi):=B$, 
respectively.
For $i\in I$ and $\Lambda\in\mbox{Irr}\Lambda(\nu)$, we set
$$\eps_i(\Lambda):=\eps_i(B)\quad\mbox{and}\quad 
\eps_i^*(\Lambda):=\eps_i^*(B),$$
where $B$ is a general point of $\Lambda$, and 
$$\eps_i(B):=\dimc\mbox{\rm Coker}\left(\mathop{\bigoplus}_{\tau;\sinn(\tau)=i}
V(\nu)_{\sout(\tau)}
\overset{\oplus B_{\tau}}{\longrightarrow}V(\nu)_i\right),$$
$$\eps_i^*(B):=\dimc\mbox{\rm Ker}\left(V(\nu)_i
\overset{\oplus B_{\tau}}{\longrightarrow}
\mathop{\bigoplus}_{\tau;\sout(\tau)=i}V(\nu)_{\sinn(\tau)}\right);$$
also for $k,l\in \nz_{\geq 0}$, we set
$$\bigl(\mbox{Irr}\Lambda(\nu)\bigr)_{i,k}:=
\{\Lambda\in \mbox{Irr}\Lambda(\nu)~|~\eps_i(\Lambda)=k\}~~\mbox{ and }~~
\bigl(\mbox{Irr}\Lambda(\nu)\bigr)_{i}^l:=
\{\Lambda\in \mbox{Irr}\Lambda(\nu)~|~\eps_i^*(\Lambda)=l\}.$$ 
Suppose now that  $\onu=c\alpha_i$ ({\it resp}., $\nu'=c\alpha_i$) for 
$c\in\nz_{\geq 0}$. Since $\Lambda(c\alpha_i)=\{0\}$, we have the following 
diagrams as special cases of (2.5.1): 
$$\Lambda(\nu')\cong\Lambda(\nu')\times\Lambda(c\alpha_i)
\overset{q_1}{\longleftarrow}
\Lambda(\nu',c\alpha_i)\overset{q_2}{\longrightarrow}\Lambda(\nu),
\eqno{(2.5.2)}$$
$$\Lambda(\onu)\cong\Lambda(c\alpha_i)\times\Lambda(\onu)
\overset{q_1}{\longleftarrow}
\Lambda(c\alpha_i,\onu)\overset{q_2}{\longrightarrow}\Lambda(\nu).
\eqno{(2.5.3)}$$
Diagrams (2.5.2) and (2.5.3) induce bijections 
$$\te_i^{max}:\bigl(\mbox{\rm Irr}\Lambda(\nu)\bigr)_{i,c}
\overset{\sim}{\to}
\bigl(\mbox{\rm Irr}\Lambda(\nu')\bigr)_{i,0}
\quad\mbox{and}\quad
\te_i^{\ast max}:\bigl(\mbox{\rm Irr}\Lambda(\nu)\bigr)_{i}^c
\overset{\sim}{\to}
\bigl(\mbox{\rm Irr}\Lambda(\onu)\bigr)_{i}^0,$$
respectively. Then, we define maps 
$$\te_i,\te_i^*:\bigsqcup_{\nu\in Q_+}\mbox{Irr}\Lambda(\nu)\to
\bigsqcup_{\nu\in Q_+}\mbox{Irr}\Lambda(\nu)\sqcup\{0\}\quad\mbox{and}\quad
\tf_i,\tf_i^*:\bigsqcup_{\nu\in Q_+}\mbox{Irr}\Lambda(\nu)\to
\bigsqcup_{\nu\in Q_+}\mbox{Irr}\Lambda(\nu)$$
as follows. If $c>0$, then 
$$\begin{array}{llllllllllll}
\te_i:& \bigl(\mbox{\rm Irr}\Lambda(\nu)\bigr)_{i,c} & 
\overset{\sim}{\longrightarrow} & 
\bigl(\mbox{\rm Irr}\Lambda(\nu')\bigr)_{i,0} & 
\overset{\sim}{\longrightarrow} & 
\bigl(\mbox{\rm Irr}\Lambda(\nu+\alpha_i)\bigr)_{i,c-1},\\
\te_i^*:& \bigl(\mbox{\rm Irr}\Lambda(\nu)\bigr)_i^c & 
\overset{\sim}{\longrightarrow} & 
\bigl(\mbox{\rm Irr}\Lambda(\onu)\bigr)_i^0 & 
\overset{\sim}{\longrightarrow} & 
\bigl(\mbox{\rm Irr}\Lambda(\nu+\alpha_i)\bigr)_i^{c-1};
\end{array}$$
and $\te_i\Lambda=0$, $\te_i^*\Lambda'=0$ for 
$\Lambda\in\bigl(\mbox{\rm Irr}\Lambda(\nu)\bigr)_{i,0}$,
$\Lambda'\in\bigl(\mbox{\rm Irr}\Lambda(\nu)\bigr)_{i}^0$, respectively.
Also, we define 
$$\begin{array}{llllllllllll}
\tf_i:& \bigl(\mbox{\rm Irr}\Lambda(\nu)\bigr)_{i,c} & 
\overset{\sim}{\longrightarrow} & 
\bigl(\mbox{\rm Irr}\Lambda(\nu')\bigr)_{i,0} & 
\overset{\sim}{\longrightarrow} & 
\bigl(\mbox{\rm Irr}\Lambda(\nu-\alpha_i)\bigr)_{i,c+1},\\
\tf_i^*:& \bigl(\mbox{\rm Irr}\Lambda(\nu)\bigr)_i^c & 
\overset{\sim}{\longrightarrow} & 
\bigl(\mbox{\rm Irr}\Lambda(\onu)\bigr)_i^0 & 
\overset{\sim}{\longrightarrow} & 
\bigl(\mbox{\rm Irr}\Lambda(\nu-\alpha_i)\bigr)_i^{c+1}.
\end{array}$$
Let $\ast : B\mapsto {}^tB$ be an automorphism of $X(\nu)$, where ${}^tB$
is the transpose of $B\in X(\nu)$. Then, $\Lambda(\nu)$ is stable under 
$\ast$, and it induces an automorphism of $\mbox{Irr}{\Lambda(\nu)}$.

\begin{lemma}[\cite{KS}]\label{lemma:KS}
We have $\te_i^*=\ast\circ\te_i\circ\ast$ and  $\tf_i^*=\ast\circ\tf_i\circ
\ast$.
\end{lemma}

\begin{thm}[\cite{KS}]\label{thm:KS}
{\rm (1)} 
For $\Lambda\in \mbox{\em Irr}\Lambda(\nu)$, we set
$\mbox{\rm wt}\Lambda:=-\nu$, $\vphi_i(\Lambda):=\eps_i(\Lambda)+
\langle h_i,\mbox{\rm wt}\Lambda\rangle$.
Then,
$\left(\bigsqcup_{\nu\in Q_+}\mbox{\em Irr}\Lambda(\nu);\mbox{\rm wt}, \eps_i, 
\vphi_i, \te_i, \tf_i\right)$ is a crystal, which is isomorphic to
$\left(B(\infty);\mbox{\rm wt}, \eps_i, \vphi_i, \te_i, \tf_i\right)$.
\vskip 1mm
\noindent
{\rm (2)} Set $\vphi_i^*(\Lambda)=\eps_i^*(\Lambda)+
\langle h_i,\mbox{\rm wt}\Lambda\rangle_I$. Then, 
$\left(\bigsqcup_{\nu\in Q_+}\mbox{\em Irr}\Lambda(\nu);\mbox{\rm wt}, \eps_i^*, 
\vphi_i^*, \te_i^*, \tf_i^*\right)$ is a crystal, which is isomorphic to
$\left(B(\infty);\mbox{\rm wt}, \eps_i^*, \vphi_i^*, \te_i^*, \tf_i^*\right)$.
\end{thm}

A Maya diagram ${\bf k}\in\cM_I^{\times}$ can be written as a disjoint union 
of intervals:
$$\begin{array}{c}
{\bf k}=[s_1+1,t_1]\sqcup [s_2+1,t_2]\sqcup\cdots\sqcup[s_l+1,t_l],\\
\mbox{where }n\leq s_1<t_1<s_2<t_2<\cdots <s_l<t_l\leq n+m+1;\end{array}$$
the interval $K_p=[s_p+1,t_p]$ is called the $p$-th component of ${\bf k}$ 
for $1\leq p\leq l$. Define two subsets $\out({\bf k})$ and $\inn({\bf k})$ of 
$I$ by
$$\out({\bf k}):=\{t_p|~1\leq p\leq l\}\cap I,
\quad \inn({\bf k}):=\{s_p|~1\leq p\leq l\}\cap I.$$
Also, we define two subsets $I_t$ and $I_s$ of $I$
by
$$I_t:=\begin{cases}
\out({\bf k})\cup\{n+1,n+m\} & \mbox{if } s_1\geq n+2,~t_l=n+m+1,\\
\out({\bf k})\cup\{n+1\} & \mbox{if } s_1\geq n+2,~t_l\leq n+m,\\
\out({\bf k})\cup\{n+m\} & \mbox{if } s_1\leq n+1,~t_l=n+m+1,\\
\out({\bf k})\cup\{n+1,n+m\} & \mbox{if } s_1\geq n+1,~t_l\leq n+m,
\end{cases}$$
$$I_s:=\begin{cases}
\inn({\bf k})\cup\{n+1,n+m\} & \mbox{if } s_1=n,~t_l\leq n+m-1,\\
\inn({\bf k})\cup\{n+1\} & \mbox{if } s_1= n,~t_l\geq n+m,\\
\inn({\bf k})\cup\{n+m\} & \mbox{if } s_1\geq n+1,~t_l\leq n+m-1,\\
\inn({\bf k})\cup\{n+1,n+m\} & \mbox{if } s_1\geq n+1,~t_l\geq n+m.
\end{cases}$$
Then, there exists a unique orientation $\Omega({\bf k})$ such that
$I_t$ is identical to the set of source vertices of the quiver $(I,\Omega({\bf k}))$, 
and $I_s$ is identical to the set of sink vertices of this quiver. \\

For $B=(B_{\tau})_{\tau\in H}\in X(\nu)$, we set
$$M_{\bf k}(B):=-\dimc\mbox{Coker}
\left(\mathop{\bigoplus}_{k\in \sout({\bf k})}V(\nu)_k
\overset{\oplus B_{\mu}}{\longrightarrow}
\mathop{\bigoplus}_{l\in \sinn({\bf k})}V(\nu)_l\right),$$
where $\mu:=\tau_{i_1}\tau_{i_2}\cdots\tau_{l_q}$ is a path from 
$k\in \out({\bf k})$ to $l\in\inn({\bf k})$ under the orientation 
$\Omega({\bf k})$, and 
$B_{\mu}:=B_{\tau_{i_1}} B_{\tau_{i_2}}\cdots B_{\tau_{i_q}}$ is the 
corresponding composite of linear maps. 
For $\Lambda\in \mbox{Irr}\Lambda(\nu)$ and ${\bf k}\in\cM_I^{\times}$, define
$$M_{\bf k}(\Lambda):=M_{\bf k}(B),$$
where $B=(B_{\tau})_{\tau\in H}$ is a general point of $\Lambda$.
\begin{prop}[\cite{S}]
{\rm (1)} For each $\Lambda\in\mbox{\rm Irr}\Lambda(\nu)$, a collection 
${\bf M}(\Lambda):=(M_{\bf k}(\Lambda))_{{\bf k}\in\cM_I^{\times}}$ of integers
is an $e$-BZ datum. 
\vskip 1mm
\noindent
{\rm (2)} The map $\Psi_I:\bigsqcup_{\nu\in Q_+}\mbox{\em Irr}\Lambda(\nu)\to
\cBZ_I^e$, defined by $\Lambda\mapsto {\bf M}(\Lambda)$, gives rise to
an isomorphism of crystals
$\left(\bigsqcup_{\nu\in Q_+}\mbox{\em Irr}\Lambda(\nu);\mbox{\rm wt}, \eps_i^*, 
\vphi_i^*, \te_i^*, \tf_i^*\right)\overset{\sim}{\to}
\left(\cBZ_I^e;\mbox{\rm wt}, \eps_i^*,\vphi_i^*, \te_i^*, \tf_i^*\right)$.
\end{prop}

\subsection{BZ data associated to $\nz$}
\begin{defn}
{\rm (1)} For a given integer $r\in\nz$, a subset ${\bf k}$ of $\nz$ is 
called a Maya diagram of charge $r$ if it satisfies the following condition: 
there exist nonnegative integers $p$ and $q$ such that
$$\nz_{\leq r-p}\subset {\bf k}\subset \nz_{\leq r+q},\quad
|{\bf k}\cap\nz_{>r-p}|=p,\eqno{(2.6.1)}$$
where $|{\bf k}\cap\nz_{>r-p}|$ denotes the cardinality of the finite set 
${\bf k}\cap\nz_{>r-p}$.
We denote by $\cM_{\snz}^{(r)}$ the set of all Maya diagrams of charge $r$,
and set $\cM_{\snz}:={\bigcup_{r\in \snz}}\cM_{\snz}^{(r)}$.
\vskip 1mm
\noindent
{\rm (2)} For a Maya diagram ${\bf k}$ of charge $r$, let 
${\bf k}^c:=\nz\setminus {\bf k}$ be the complement of ${\bf k}$ in $\nz$. 
We call ${\bf k}^c$ the complementary Maya diagram of charge $r$ associated
to ${\bf k}\in \cM_{\snz}^{(r)}$. We denote by $\cM_{\snz}^{(r),c}$ 
the set of all 
complementary Maya diagrams of charge $r$,
and set $\cM_{\snz}^c:={\bigcup_{r\in \snz}}\cM_{\snz}^{(r),c}$.
\end{defn}
 
A map $c: \cM_{\snz} \to \cM_{\snz}^c$ defined by ${\bf k}\mapsto {\bf k}^c$ 
is a bijection; the inverse of this map is also denoted by $c$.

We identify $\Xi_{\snz}:=\bigsqcup_{r\in\nz}W_{\nz}\Lambda_r$ 
({\it resp}., $\Gamma_{\snz}:=\bigsqcup_{r\in\nz}W_{\nz}\Lambda_r^c$) with 
$\cM_{\snz}$ ({\it resp}., $\cM_{\snz}^c$) via the bijection 
$\Lambda_r\leftrightarrow
\nz_{\leq r}$ ({\it resp}., $\Lambda_r^c\leftrightarrow\nz_{> r}$).
Under the identification $\Xi_{\snz}\cong\cM_{\snz}$ ({\it resp}., 
$\Gamma_{\snz}\cong\cM_{\snz}^c$), there is an induced action 
of $\sigma_i\in W_{\nz}$ on $\cM_{\snz}$ ({\it resp}., $\cM_{\snz}^c$). 
It is easy to see that the explicit form of 
this action is just the transposition $(i,i+1)$ of $\nz$. For $\xi\in \Xi_{\nz}$
({\it resp}., $\gamma\in\Gamma_{\nz})$, we denote by ${\bf k}(\xi)$ ({\it resp}.,
${\bf k}(\gamma)$) the corresponding Maya diagram.

Let $I$ be a finite interval in $\nz$, and $\mbox{res}_I: \cM_{\nz}\to \cM_I$ 
a map defined by 
$\mbox{res}_I({\bf k})={\bf k}\cap\widetilde{I}$ for ${\bf k}\in\cM_{\nz}$. 
Set $\cM_{\snz}(I):=\left\{{\bf k}\in\cM_{\snz}~|~{\bf k}
=\nz_{\leq n}\cup{\bf k}_I\mbox{ for some }{\bf k}_I\in\cM_I^{\times}\right\}$. 
Then the map 
$\mbox{res}_I$ induces a bijection from $\cM_{\nz}(I)$ to $\cM_I^{\times}$.
For ${\bf k}\in \cM_{\snz}(I)$, if we set $\Omega_I({\bf k}):=
(\mbox{res}_I)^{-1}\bigl(\widetilde{I}\setminus\mbox{res}_{I}({\bf k})\bigr)$
for ${\bf k}\in \cM_{\nz}(I)$,
then $\Omega_I({\bf k})\in \cM_{\snz}(I)$ and the map 
$\Omega_I:\cM_{\snz}(I)\to\cM_{\snz}(I)$ is a bijection.
Also, if we define $\mbox{res}_I^c:=\mbox{res}_I\circ c:\cM_{\nz}^c\to \cM_I$, 
then it induces bijections  $\mbox{res}_I^c:\cM_{\nz}^c(I):=(\cM_{\nz}(I))^c
\overset{\sim}{\to}\cM_I^{\times}$ and $\Omega_I^c:\cM_{\nz}^c(I)
\overset{\sim}{\to}\cM_{\nz}^c(I)$ in a similar way.\\

Let ${\bf M}=(M_{\bf k})_{{\bf k}\in \cM_{\snz}}$ be a collection of integers
indexed by $\cM_{\snz}$. For such an ${\bf M}$, we define  
${\bf M}_I:=(M_{\bf k})_{{\bf k}\in \cM_{\snz}(I)}$. By the bijection
$\mbox{res}_I:\cM_{\snz}(I)\overset{\sim}{\to} \cM_I^{\times}$, 
${\bf M}_I$ can be regarded as a collection of integers indexed by 
$\cM_I^{\times}$. Similarly, 
for ${\bf M}=(M_{\bf k})_{{\bf k}\in \cM_{\snz}^c}$, we define 
${\bf M}_I:=(M_{\bf k})_{{\bf k}\in \cM_{\snz}^c(I)}$, which is 
regarded as a collection of integers indexed by $\cM_I^{\times}$.

\begin{defn}\label{defn:BZ}
{\em (1)} A collection ${\bf M}=(M_{\bf k})_{{\bf k}\in \cM_{\snz}^c}$ of integers 
is called a complementary BZ  {\rm (}c-BZ for short{\rm )} datum associated to 
$\nz$ if it satisfies the following conditions:
\vskip 1mm
{\rm (1-a)} For each finite interval $K$ in $\nz$, 
${\bf M}_K=(M_{\bf k})_{{\bf k}\in\cM_K^{\times}}$ is an 
element of $\mathcal{BZ}_K$.
\vskip 1mm
{\rm (1-b)} For each ${\bf k}\in \cM_{\nz}^c$, there exists a finite interval 
$I$ in $\nz$ such that
\vskip 1mm
\hspace*{5mm}{\rm (1-i)} ${\bf k}\in\cM_{\snz}^c(I)$,
\vskip 1mm
\hspace*{4mm}{\rm (1-ii)} for every finite interval $J\supset I$, 
${M}_{\Omega_J^c(\bf k)}={M}_{\Omega_I^c(\bf k)}$.
\vskip 1mm
\noindent
{\rm (2)} A collection ${\bf M}=(M_{\bf k})_{{\bf k}\in \cM_{\snz}}$ of integers 
is called an $e$-BZ datum associated to $\nz$ if it satisfies the following 
conditions:
\vskip 1mm
{\rm (2-a)} For each finite interval $K$ in $\nz$, 
${\bf M}_K=(M_{\bf k})_{{\bf k}\in\cM_K^{\times}}$ is an 
element of $\mathcal{BZ}_K^{e}$.
\vskip 1mm
{\rm (2-b)} For each ${\bf k}\in \cM_{\snz}$, there exists a finite interval 
$I$ in $\nz$ such that 
\vskip 1mm
\hspace*{5mm}{\rm (2-i)} ${\bf k}\in\cM_{\snz}(I)$,
\vskip 1mm
\hspace*{4mm}{\rm (2-ii)} for every finite interval $J\supset I$, 
${M}_{\Omega_J(\bf k)}={M}_{\Omega_I(\bf k)}$.
\end{defn}
We denote by $\mathcal{BZ}_{\snz}$ ({\it resp}., $\mathcal{BZ}_{\snz}^e$) the
set of all c-BZ ({\it resp}., $e$-BZ) data associated to $\nz$.\\

For a given c-BZ datum ${\bf M}=(M_{\bf k})_{{\bf k}\in \cM_{\snz}^c}\in 
\mathcal{BZ}_{\snz}$, we define a new collection  
${\bf M}^*=(M^*_{\bf k})_{{\bf k}\in \cM_{\snz}}$ of integers by 
$M^*_{\bf k}:=M_{{\bf k}^c}.$
As in the case of finite intervals, ${\bf M}^*$ is an element of 
$\mathcal{BZ}_{\snz}^e$ and the map $\ast:\mathcal{BZ}_{\snz}\to
\mathcal{BZ}_{\snz}^e$ is a bijection. 
The inverse of this bijection is also denoted by $\ast$. We note that $\ast^2=
\mbox{id}$.\\

Let ${\bf M}=(M_{\bf k})_{{\bf k}\in \cM_{\snz}^c}\in \mathcal{BZ}_{\snz}$ be a 
c-BZ datum. For each complementary Maya diagram ${\bf k}\in 
\cM_{\snz}^c$, we denote by $\mbox{Int}^c({\bf M};{\bf k})$ the set of all 
finite intervals $I$ in $\nz$ satisfying condition (1-b) in the definition
above. 

For ${\bf M}\in \mathcal{BZ}_{\snz}$, we define another collection 
$\Theta({\bf M})=(\Theta(M)_{\bf k})_{{\bf k}\in \cM_{\snz}}$ of integers 
as follows. Fix ${\bf k}\in\cM_{\snz}$ and take the complement 
${\bf k}^c\in\cM_{\snz}^c$ of ${\bf k}$. 
Since ${\bf M}\in \mathcal{BZ}_{\snz}$, there exists a finite 
interval $I\in \mbox{Int}^c({\bf M};{\bf k}^c)$. Then we define 
$\Theta(M)_{\bf k}:=M_{(\mbox{res}_I^c)^{-1}(\mbox{res}_I({\bf k}))}$;
this definition does not depend on the choice of $I$. \\

Now, let ${\bf M}=(M_{\bf k})_{{\bf k}\in \cM_{\snz}}\in \cBZ_{\nz}^e$.
Note that ${\bf M}^*\in \cBZ_{\nz}$. We set 
$$\mbox{Int}^e({\bf M};{\bf k}):=\mbox{Int}^c({\bf M}^*;{\bf k}^c).
\eqno{(2.6.2)}$$
\begin{lemma}\label{lemma:int^e}
The set $\mbox{\rm Int}^e({\bf M};{\bf k})$ is identical to the set of all finite 
intervals $I$ in $\nz$ satisfying condition (2-b) in Definition \ref{defn:BZ}. 
\end{lemma}

\begin{proof}
It suffices to show that $I\in \mbox{Int}^e({\bf M};{\bf k})=
\mbox{Int}^c({\bf M}^*;{\bf k}^c)$ if and only if $I$ satisfies 
condition (2-b). By the definition of $\cM_{\nz}^c(I)$, the condition 
${\bf k}^c\in \cM_{\nz}^c(I)$ is equivalent to the condition 
${\bf k}\in \cM_{\nz}(I)$. 
Suppose that  this condition is satisfied. Recall the following equation 
in Lemma 3.3.1 of \cite{NSS2}: 
$\Omega_I^c({\bf k}^c)=\bigl(\Omega_I({\bf k})\bigr)^c$ for 
${\bf k}\in \cM_{\nz}(I)$. 
From this, we deduce that
$$M^*_{\Omega_I^c({\bf k}^c)}=M_{\bigl(\Omega_I^c({\bf k}^c)\bigr)^c}
=M_{\Omega_I({\bf k})}.$$
Thus, condition (1-ii) for ${\bf M}^*$ and ${\bf k}^c$ is
equivalent to condition (2-ii) for ${\bf M}$ and ${\bf k}$.
This proves the lemma.
\end{proof} 
\subsection{Action of Kashiwara operators}
First, we define the action of the raising Kashiwara operators 
$\te_p$, $p\in\nz$, on $\mathcal{BZ}_{\snz}$. For ${\bf M}=
(M_{\bf k})_{{\bf k}\in \cM_{\snz}^c}\in \mathcal{BZ}_{\snz}$ and $p\in\nz$, 
set
$$\eps_p({\bf M}):=-\left(\Theta(M)_{{\bf k}(\Lambda_p)}+
\Theta(M)_{{\bf k}(\sigma_p\Lambda_p)}-\Theta(M)_{{\bf k}(\Lambda_{p+1})}
-\Theta(M)_{{\bf k}(\Lambda_{p-1})}\right).$$
Let $I\in \mbox{Int}^c({\bf M};{\bf k}(\Lambda_p)^c)\cap 
\mbox{Int}^c({\bf M};{\bf k}(\sigma_p\Lambda_p)^c)\cap
\mbox{Int}^c({\bf M};{\bf k}(\Lambda_{p+1})^c)\cap
\mbox{Int}^c({\bf M};{\bf k}(\Lambda_{p-1})^c).$ Then it is known that
$\eps_p({\bf M})=
\eps_p({\bf M}_I)$, and hence this is a nonnegative integer. 
 
If $\eps_p({\bf M})=0$, then we set $\te_p{\bf M}={0}$. Suppose that 
$\eps_p({\bf M})>0$. Then we define $\te_p{\bf M}=
(M_{\bf k}')_{{\bf k}\in \cM_{\snz}^c}$ as follows. For ${\bf k}\in \cM_{\snz}^c$,
take a finite interval $I$ in $\nz$ such that
${\bf k}\in\cM_{\snz}^c(I)$ and 
$I\in\mbox{Int}^c({\bf M};{\bf k}(\Lambda_p)^c)\cap 
\mbox{Int}^c({\bf M};{\bf k}(\sigma_p\Lambda_p)^c)\cap
\mbox{Int}^c({\bf M};{\bf k}(\Lambda_{p+1})^c)\cap
\mbox{Int}^c({\bf M};{\bf k}(\Lambda_{p-1})^c).$
Set
$$M'_{\bf k}:=\bigl(\te_p{\bf M}_I\bigr)_{\mbox{res}_I^c({\bf k})}.$$
Here we note that $\te_p{\bf M}_I$ is defined since 
${\bf M}_I\in\mathcal{BZ}_I$.\\

Second, let us define the action of the lowering Kashiwara operators 
$\tf_p,~p\in\nz,$ on $\mathcal{BZ}_{\snz}$. For ${\bf M}=
(M_{\bf k})_{{\bf k}\in \cM_{\snz}^c}\in \mathcal{BZ}_{\snz}$ and $p\in\nz$,
we define $\tf_p{\bf M}=
(M_{\bf k}'')_{{\bf k}\in \cM_{\snz}^c}$ as follows. For ${\bf k}\in \cM_{\snz}^c$,
take a finite interval $I$ in $\nz$ such that
${\bf k}\in\cM_{\snz}^c(I)$ and 
$I\in \mbox{Int}^c({\bf M};{\bf k}(\Lambda_p)^c)\cap 
\mbox{Int}^c({\bf M};{\bf k}(\sigma_p\Lambda_p)^c).$ 
Set
$$M''_{\bf k}:=\bigl(\tf_p{\bf M}_I\bigr)_{\mbox{res}_I^c({\bf k})}.$$

\begin{prop}[\cite{NSS}]
{\rm (1)} The definition above of $M'_{\bf k}$ {\rm (}{\it resp}., 
$M''_{\bf k}${\rm )} does not depend
on the choice of $I$. 
\\
{\rm (2)} For each ${\bf M}=
(M_{\bf k})_{{\bf k}\in \cM_{\snz}^c}\in \mathcal{BZ}_{\snz}$ and $p\in\nz$,
$\te_p{\bf M}$ {\rm (}{\it resp}., $\tf_p{\bf M}${\rm )} is contained in 
$\mathcal{BZ}_{\snz}\cup\{0\}$ {\rm (}{\it resp}., 
$\mathcal{BZ}_{\snz}${\rm )}.
\end{prop} 

For ${\bf M}\in \cBZ_{\snz}^e$, set $\eps_p^*({\bf M}):=\eps_p({\bf M}^*),~
p\in\nz$. We define the Kashiwara operators $\te_p^*$ and $\tf_p^*$ on 
$\cBZ_{\snz}^e$ by
$$\te_p^*{\bf M}:=\begin{cases}
\bigl(\te_p({\bf M}^*)\bigr)^*& \mbox{if }\eps_p^*({\bf M})>0,\\
0 & \mbox{if }\eps_p^*({\bf M})=0,
\end{cases}\quad\mbox{and}\quad 
\tf_p^*{\bf M}:=\bigl(\tf_p({\bf M}^*)\bigr)^*.$$
The following corollary is easily obtained from the proposition above.
\begin{cor}
For each ${\bf M}\in\mathcal{BZ}_{\snz}^e$ and $p\in\nz$,
$\te_p^*{\bf M}$ {\rm (}{\it resp}., $\tf_p^*{\bf M}${\rm )} is contained in 
$\mathcal{BZ}_{\snz}^e\cup\{0\}$ {\rm (}{\it resp}., 
$\mathcal{BZ}_{\snz}^e${\rm )}.
\end{cor}
\subsection{BZ data of type $A_{l-1}^{(1)}$}
Fix $l\in \nz_{\geq 3}$. Let $\widehat{\gtg}$ be the affine Lie algebra of 
type $A_{l-1}^{(1)}$, 
$\widehat{\gth}$ the Cartan subalgebra of $\widehat{\gtg}$,
$\widehat{h}_i\in\widehat{\gth},~i\in\widehat{I}:=\{0,1,\cdots,l-1\}$, 
the simple coroots of $\widehat{\gtg}$, and 
$\widehat{\alpha}_i\in \widehat{\gth}^*:=
\homc(\widehat{\gth},\nc),~i\in\widehat{I}$, the simple roots of 
$\widehat{\gtg}$. 
We set $\widehat{Q}^+:=\sum_i\nz_{\geq 0}\widehat{\alpha}_i$
and $\widehat{Q}^-:=-\widehat{Q}^+$.
Note that $\langle \widehat{h}_i,\widehat{\alpha}_j\rangle
=\widehat{a}_{ij}$ for $i,j\in\widehat{I}$. Here, $\langle\cdot,\cdot\rangle:
\widehat{\gth}\times \widehat{\gth}^*\to \nc$ is the canonical pairing, and
$\widehat{A}=(\widehat{a}_{ij})_{i,j\in\widehat{I}}$ is the Cartan matrix
of type $A_{l-1}^{(1)}$  with index set $\widehat{I}$; the
entries $\widehat{a}_{ij}$ are given by
$$\widehat{a}_{ij}:=\left\{\begin{array}{ll}
2 & \mbox{if }i=j,\\
-1 & \mbox{if }|i-j|=1\mbox{ or }l-1,\\
0 & \mbox{otherwise}.
\end{array}\right.$$ 

Now, consider a bijection $\tau:\nz\to\nz$ given by $\tau(j):=j+1$ for 
$j\in\nz$. It induces an automorphism $\tau:\gtt^*\overset{\sim}{\to}\gtt^*$
such that $\tau(\Lambda_j)=\Lambda_{j+1}$ and $\tau(\Lambda_j^c)=
\Lambda_{j+1}^c$ for all $j\in\nz$. 
It follows that $\tau\circ \sigma_j=\sigma_{j+1}\circ \tau$. 
Also, for $i\in \widehat{I}$, define a family $S_i$ of automorphism of 
$\gtt^*$ by
$$S_i:=\{\sigma_{i+al}~|~a\in \nz\}.$$
Since $l\geq 3$,  
$\sigma_{j_1}\sigma_{j_2}=\sigma_{j_2}\sigma_{j_1}$
for all $\sigma_{j_1},\sigma_{j_2}\in S_i$, and for a fixed ${\bf k}\in \cM_{\nz}$
or $\cM_{\snz}^c$, there exists a finite subset $S_i({\bf k})\subset S_i$ 
such that $\sigma_j({\bf k})={\bf k}$ for every$\sigma_j\in S_i\setminus 
S_i({\bf k})$. 
Therefore, we can define an infinite product 
$\widehat{\sigma}_i:=\prod_{\sigma_j\in S_i}\sigma_j$ of operators acting on 
$\cM_{\nz}$ and $\cM_{\snz}^c$. Note that we have 
$\tau \circ \widehat{\sigma}_i
=\widehat{\sigma}_{i+1}\circ \tau$, where we regard $i\in \widehat{I}$ as an 
element of $\nz/l\nz$. 

Set $\sigma:=\tau^l$. For ${\bf M}\in \cBZ_{\snz}$, we define new 
collections $\sigma({\bf M})$ and 
$\sigma^{-1}({\bf M})$ of integers indexed by $\cM_{\snz}^c$ by
$\sigma({\bf M})_{{\bf k}}:={\bf M}_{\sigma^{-1}({\bf k})}$ and 
$\sigma^{-1}({\bf M})_{{\bf k}}:={\bf M}_{\sigma({\bf k})}$ for each
${\bf k}\in \cM_{\snz}^c$, respectively. It is shown in \cite{NSS} that
$\sigma({\bf M})$ and $\sigma^{-1}({\bf M})$ are both elements of 
$\cBZ_{\snz}$. 

Similarly, for ${\bf M}\in \cBZ_{\snz}^{e}$, we can
define new collections $\sigma^{\pm}({\bf M})$, and prove that they are
both elements of $\cBZ_{\snz}^{e}$.
\begin{lemma}[\cite{NSS}]
{\rm (1)} On $\cBZ_{\snz}$, we have $\Theta\circ \sigma=\sigma\circ\Theta$.
\vskip 1mm
\noindent
{\rm (2)} For ${\bf M}\in \cBZ_{\snz}$ and $p\in\nz$, 
$\eps_p(\sigma({\bf M}))=\eps_{\sigma^{-1}(p)}({\bf M})$.
\vskip 1mm
\noindent
{\rm (3)} The equalities $\sigma\circ\te_p=\te_{\sigma(p)}\circ\sigma$ and 
$\sigma\circ\tf_p=\tf_{\sigma(p)}\circ\sigma$ hold on $\cBZ_{\snz}\cup\{{0}\}$
for all $p\in\nz$. Here it is understood that $\sigma({0})={0}$.
\end{lemma}
\begin{defn}\label{defn:affineBZ}
Set
$$\cBZ_{\snz}^{\sigma}:=\{{\bf M}\in \cBZ_{\snz}~|~
\sigma({\bf M})={\bf M}\}\quad\mbox{and}\quad
(\cBZ_{\snz}^e)^{\sigma}:=\{{\bf M}\in \cBZ_{\snz}^e~|~
\sigma({\bf M})={\bf M}\}.$$
An element ${\bf M}$ of $\cBZ_{\snz}^{\sigma}$ {\rm (}{\it resp}., 
$(\cBZ_{\snz}^e)^{\sigma}${\rm )}  
is called a $c$-BZ {\rm (}{\it resp}., 
$e$-BZ{\rm )} datum of type $A_{l-1}^{(1)}$.
\end{defn}
\subsection{Crystal structure on $\cBZ_{\nz}^{\sigma}$}
Now we define a crystal structure on $\cBZ_{\snz}^{\sigma}$, 
following \cite{NSS}. For ${\bf M}\in \cBZ_{\snz}^{\sigma}$ and $p\in
\widehat{I}$, we set
$$\mbox{wt}({\bf M}):=\sum_{p\in\widehat{I}}\Theta({\bf M})_{{\bf k}(\Lambda_p)}
\widehat{\alpha}_p,\quad \widehat{\eps}_p({\bf M}):=\eps_p({\bf M}),\quad
\widehat{\vphi}_p({\bf M}):=\widehat{\eps}_p({\bf M})+
\langle \widehat{h}_p,\mbox{wt}({\bf M})\rangle.$$

In order to define the action of Kashiwara operators, we need the following.
\begin{lemma}[\cite{NSS}]\label{lemma:L}
Let $q,q'\in \nz$, with $|q-q'|\geq 2$. Then, we have $\te_q\te_{q'}
=\te_{q'}\te_q$, $\tf_q\tf_{q'}=\tf_{q'}\tf_q$, and $\te_q\tf_{q'}=\tf_{q'}\te_q$, 
as operators on $\cBZ_{\snz}\cup\{{0}\}$.
\end{lemma}
For ${\bf M}\in \cBZ_{\snz}^{\sigma}$ and $p\in\widehat{I}$, 
we define $\hte_p{\bf M}$ and $\htf_p{\bf M}$ as follows.
If $\widehat{\eps}_p({\bf M})=0$, then we set $\hte_p{\bf M}:={0}$. If
$\widehat{\eps}_p({\bf M})>0$, then we define a new collection $\hte_p{\bf M}=
(M'_{\bf k})$ of integers indexed by $\cM_{\snz}^c$ by 
$$M'_{\bf k}:=\bigl(e_{L({\bf k},p)}{\bf M}\bigr)_{\bf k}\quad\mbox{for each }
{\bf k}\in \cM_{\snz}^c.$$
Here, $L({\bf k},p):=\{q\in p+l\nz~|~q\in {\bf k}\mbox{ and }q+1\not\in 
{\bf k}\}$ and 
$e_{L({\bf k},p)}:=\prod_{q\in L({\bf k},p)}\te_q$. By the definition, 
$L({\bf k},p)$ is a finite set such that $|q-q'|>2$ for all $q,q'\in 
L({\bf k},p)$ with $q\ne q'$. Therefore, by Lemma \ref{lemma:L}, 
$e_{L({\bf k},p)}$ is a well-defined operator on $\cBZ_{\snz}$. 

A collection $\htf_p{\bf M}=(M''_{\bf k})$ of integers indexed by $\cM_{\snz}^c$ 
is defined by 
$$M''_{\bf k}:=\bigl(f_{L({\bf k},p)}{\bf M}\bigr)_{\bf k}\quad\mbox{for each }
{\bf k}\in \cM_{\snz}^c,$$
where $f_{L({\bf k},p)}:=\prod_{q\in L({\bf k},p)}\tf_q$. By the same reasoning
as above, we see that $f_{L({\bf k},p)}$ is a well-defined operator on 
$\cBZ_{\snz}$.
\begin{prop}[\cite{NSS}]
{\rm (1)} We have $\hte_p{\bf M}\in \cBZ_{\snz}^{\sigma}\cup\{{0}\}$ 
and $\htf_p{\bf M}\in \cBZ_{\snz}^{\sigma}$.
\vskip 1mm
\noindent
{\rm (2)} The set $\cBZ_{\snz}^{\sigma}$, equipped with the maps
$\mbox{\rm wt},\widehat{\eps}_p,\widehat{\vphi}_p,\hte_p,\htf_p$, is a 
$U_q(\widehat{\gtsl}_l)$-crystal.
\end{prop}

Let ${\bf O}$ be a collection of integers indexed by $\cM_{\snz}^c$ whose
${\bf k}$-component is equal to $0$ for all ${\bf k}\in \cM_{\snz}^c$. 
It is obvious that ${\bf O}\in \cBZ_{\snz}^{\sigma}$. Let 
$\cBZ_{\snz}^{\sigma}({\bf O})$ denote the connected component of the
crystal $\cBZ_{\snz}^{\sigma}$ containing ${\bf O}$. The following is
the main result of \cite{NSS}.
\begin{thm}[\cite{NSS}]\label{thm:NSS-main}
As a crystal, $\left(\cBZ_{\snz}^{\sigma}({\bf O});\mbox{\rm wt},
\widehat{\eps}_p,\widehat{\vphi}_p,\hte_p,\htf_p\right)$ is isomorphic to
$B(\infty)$ for $U_q(\widehat{\gtsl}_l)$.
\end{thm}

In a manner similar to the one in \cite{NSS}, we can define a crystal
structure on 
$(\cBZ_{\snz}^{e})^{\sigma}$. By the construction, it is easy to see that 
$\ast\circ \sigma=\sigma\circ\ast$. Therefore, the restriction of 
$\ast:\cBZ_{\snz}\overset{\sim}{\to}\cBZ_{\snz}^e$ to
the subset $\cBZ_{\snz}^{\sigma}$ gives rise to a bijection
$\ast:\cBZ_{\snz}^{\sigma}\overset{\sim}{\to}(\cBZ_{\snz}^e)^{\sigma}$.
We denote by ${\bf O}^*$ the image of ${\bf O}\in \cBZ_{\snz}^{\sigma}$
under the bijection $\ast$. Then, ${\bf O}^*$ is  
a collection of integers indexed by $\cM_{\snz}$ whose
${\bf k}$-component is equal to $0$ for all ${\bf k}\in \cM_{\snz}$.

For ${\bf M}\in(\mathcal{BZ}_{\snz}^e)^{\sigma}$ and $p\in\nz$, we define
$$\mbox{wt}({\bf M}):=\mbox{wt}({\bf M}^*),\quad 
\widehat{\eps}_p^*({\bf M}):=\widehat{\eps}_p({\bf M}^*),\quad
\widehat{\vphi}_p^*({\bf M}):=\widehat{\eps}_p^*({\bf M})+
\langle \widehat{h}_p,\mbox{wt}({\bf M})\rangle,$$
and
$$\hte_p^*{\bf M}:=\left\{\begin{array}{ll}
(\hte_p({\bf M}^*))^* & \mbox{if }~\widehat{\eps}_p^*({\bf M})>0,\\
{0} & \mbox{if }~\widehat{\eps}_p^*({\bf M})=0,
\end{array}\right.\qquad
\htf_p^*:=(\htf_p({\bf M}^*))^*.$$
The following corollary is an easy consequence of Theorem \ref{thm:NSS-main}.
\begin{cor}
{\rm (1)} The set $(\cBZ_{\snz}^{e})^{\sigma}$, equipped with the maps
$\mbox{\rm wt},\widehat{\eps}_p^*,\widehat{\vphi}_p^*,\hte_p^*,\htf_p^*$, is 
a $U_q(\widehat{\gtsl}_l)$-crystal.\\
{\rm (2)} Let
$(\cBZ_{\snz}^{e})^{\sigma}({\bf O}^*)$ be the connected component of the
crystal $(\cBZ_{\snz}^{e})^{\sigma}$ containing ${\bf O}^*\in 
(\cBZ_{\snz}^{e})^{\sigma}$. Then, 
$\left((\cBZ_{\snz}^{e})^{\sigma}({\bf O}^*);\mbox{\rm wt},\widehat{\eps}_p^*,
\widehat{\vphi}_p^*,\hte_p^*,\htf_p^*\right)$ is isomorphic as a crystal to
$B(\infty)$ for $U_q(\widehat{\gtsl}_l)$.
\end{cor}
\section{Ordinary crystal structure on $\cBZ_I^e$}
\subsection{The operator $\sharp$}
Let ${\bf M}=(M_{\bf k})_{{\bf k}\in \cM_I^{\times}}$ be an $e$-BZ datum
associated to a finite interval $I=[n+1,n+m]$. Set 
$\mbox{wt}^{\vee}({\bf M}):=\sum_{i\in I}M_{[i+1,n+m+1]}h_i$. Then the
following equality holds: 
$$\langle \mbox{wt}^{\vee}({\bf M}),\alpha_i^I\rangle_I
=\langle h_i,\mbox{wt}({\bf M})\rangle_I.$$
\begin{defn}
For each ${\bf M}=(M_{\bf k})_{{\bf k}\in \cM_I}\in \cBZ_I^e$, we define
a new collection ${\bf M}^{\sharp}=(M_{\bf k}^{\sharp})_{{\bf k}\in 
\cM_I^{\times}}$ of integers by
$$M_{\bf k}^{\sharp}:=M_{{\bf k}^c}-\langle \mbox{\rm wt}^{\vee}({\bf M}),
{\bf k}\rangle_I.$$ 
\end{defn}
It is easy to verify the following lemma.
\begin{lemma}\label{lemma:sharp1}
{\em (1)} $\mbox{\rm wt}({\bf M}^{\sharp})=
\mbox{\rm wt}({\bf M})$ and 
$\mbox{\rm wt}^{\vee}({\bf M}^{\sharp})=
\mbox{\rm wt}^{\vee}({\bf M})$.
\vskip 1mm
\noindent
{\rm (2)} $({\bf M}^{\sharp})^{\sharp}={\bf M}.$
\end{lemma}

\begin{lemma}\label{lemma:sharp2}
${\bf M}^{\sharp}\in \cBZ_I^e.$
\end{lemma}
\begin{proof}
It suffices to check conditions (BZ-0), (BZ-1), and (BZ-2). Condition
(BZ-0) is checked by an easy calculation. 
Let us check (BZ-1): for ${\bf k}\cap\{i,j\}=\phi$,
$$M_{{\bf k}\cup \{i\}}^{\sharp}+M_{{\bf k}\cup \{j\}}^{\sharp}\leq
M_{{\bf k}\cup \{i,j\}}^{\sharp}+M_{{\bf k}}^{\sharp}.$$ 
Since
$$\langle h_k,{\bf k}\cup\{i\}\rangle_I=\begin{cases}
1 & \mbox{if }k=i,\\
-1 & \mbox{if }k=i-1,\\
\langle h_k,{\bf k}\rangle_I & \mbox{otherwise},
\end{cases}
\quad
\langle h_k,{\bf k}\cup\{j\}\rangle_I=\begin{cases}
1 & \mbox{if }k=j,\\
-1 & \mbox{if }k=j-1,\\
\langle h_k,{\bf k}\rangle_I & \mbox{otherwise},
\end{cases}
$$
and
$$\langle h_k,{\bf k}\cup\{i,j\}\rangle_I=\begin{cases}
1 & \mbox{if }k=i\mbox{ or }j,\\
-1 & \mbox{if }k=i-1\mbox{ or }j-1,\\
\langle h_k,{\bf k}\rangle_I & \mbox{otherwise},
\end{cases}
$$
we obtain the following equalities:
$$\langle h_k,{\bf k}\cup\{i\}\rangle_I+\langle h_k,{\bf k}\cup\{j\}\rangle_I
=\langle h_k,{\bf k}\cup\{i,j\}\rangle_I+\langle h_k,{\bf k}\rangle_I\quad
\mbox{for all }k\in I.$$ 
From this, we deduce that
$$\langle\mbox{wt}^{\vee}({\bf M}),{\bf k}\cup\{i\}\rangle_I+
\langle\mbox{wt}^{\vee}({\bf M}),{\bf k}\cup\{j\}\rangle_I=
\langle\mbox{wt}^{\vee}({\bf M}),{\bf k}\cup\{i,j\}\rangle_I+
\langle\mbox{wt}^{\vee}({\bf M}),{\bf k}\rangle_I.$$
Since  $M_{({\bf k}\cup \{i\})^c}+M_{({\bf k}\cup \{j\})^c}\leq
M_{({\bf k}\cup \{i,j\})^c}+M_{{\bf k}^c}$, condition (BZ-1) is satisfied for 
${\bf M}^{\sharp}$.

Now, suppose that ${\bf k}\cap\{i,j,k\}=\phi$ with $i<j<k$. 
Then, for every $l\in I$, we have
\begin{align*}
&\langle h_l,{\bf k}\cup\{i,k\}\rangle_I+\langle h_l,{\bf k}\cup\{j\}\rangle_I
=\langle h_l,{\bf k}\cup\{j,k\}\rangle_I+\langle h_l,{\bf k}\rangle_I
=\langle h_l,{\bf k}\cup\{i,j\}\rangle_I+\langle h_l,{\bf k}\rangle_I.
\end{align*}
From this equality, we see that condition (BZ-2) is satisfied for 
${\bf M}^{\sharp}$ by the same argument as for condition (BZ-1).
\end{proof}
For ${\bf M}\in\cBZ_I^e$, set $\eps_i({\bf M}):=-M_{[n+1,i-1]\cup\{i+1\}}$.
\begin{lemma}\label{lemma:sharp3}
$\eps_i({\bf M})=\eps_i^*({\bf M}^{\sharp})$.
\end{lemma}

\begin{proof}
By Lemma \ref{lemma:sharp1} (2), it suffices to show that 
$\eps_i({\bf M}^{\sharp})=\eps_i^*({\bf M})$. By the definitions, we have
$$\eps_i^*({\bf M})=
-M_{[i+1,n+m+1]}-M_{\{i\}\cup[i+2,n+m+1]}+M_{[i+2,n+m+1]}+M_{[i,n+m+1]}.$$
Also, we compute:
\begin{align*}
\eps_i({\bf M}^{\sharp})
&= -M^{\sharp}_{[n+1,i-1]\cup\{i+1\}}\\
&= -M_{([n+1,i-1]\cup\{i+1\})^c}+\langle \mbox{wt}^{\vee}({\bf M}),
[n+1,i-1]\cup\{i+1\}\rangle_I\\
&=-M_{\{i\}\cup[i+2,n+m+1]}+\sum_{l\in I}M_{[l+1,n+m+1]}\langle h_l,
[n+1,i-1]\cup\{i+1\}\rangle_I\\
&=-M_{\{i\}\cup[i+2,n+m+1]}-M_{[i+1,n+m+1]}+M_{[i+2,n+m+1]}+M_{[i,n+m+1]}.
\end{align*}
Thus, we obtain the desired equality.
\end{proof}

\begin{lemma}\label{lemma:sharp4}
{\rm (1)} If $\eps_i({\bf M})>0$, then
\vskip 1mm
{\rm (a)} $\bigl(\te_i^*({\bf M}^{\sharp})\bigr)^{\sharp}_{\bf k}
={M}_{\bf k}+1$ for ${\bf k}\in \cM_I^{\times}(i)^*$,
\vskip 1mm
{\rm (b)} $\left(\te_i^*({\bf M}^{\sharp})\right)^{\sharp}_{\bf k}
={M}_{\bf k}$ for ${\bf k}\in \cM_I^{\times}\setminus
\left(\cM_I^{\times}(i)\cup\cM_I^{\times}(i)^*\right)$.
\vskip 1mm
\noindent
{\rm (2)} For every ${\bf M}\in \cBZ_I^e$, 
\vskip 1mm
{\rm (a)} $\bigl(\tf_i^*({\bf M}^{\sharp})\bigr)^{\sharp}_{\bf k}
={M}_{\bf k}-1$ for ${\bf k}\in \cM_I^{\times}(i)^*$,
\vskip 1mm
{\rm (b)} $\left(\tf_i^*({\bf M}^{\sharp})\right)^{\sharp}_{\bf k}
={M}_{\bf k}$ for ${\bf k}\in \cM_I^{\times}\setminus
\left(\cM_I^{\times}(i)\cup\cM_I^{\times}(i)^*\right)$.
\end{lemma}
\begin{proof}
Since part (2) is proved in a similar way, we only give a proof of part (1).
Suppose that ${\bf k}\in \cM_I^{\times}(i)^*$ or ${\bf k}\in\cM_I^{\times}
\setminus\left(\cM_I^{\times}(i)\cup\cM_I^{\times}(i)^*\right)$. Then, 
${\bf k}^c\in \cM_I^{\times}\setminus \cM_I^{\times}(i)^*$. Also, 
since $\eps_i^*({\bf M}^{\sharp})=\eps_i({\bf M})>0$, it follows that
\begin{align*}
\te_i^*({\bf M}^{\sharp})_{{\bf k}^c}
&={\bf M}^{\sharp}_{{\bf k}^c}\qquad
\mbox{by Proposition \ref{prop:ast-action}}\\
&={M}_{\bf k}-\langle \mbox{wt}^{\vee}({\bf M}),{\bf k}^c\rangle_I\\
&={M}_{\bf k}+\langle \mbox{wt}^{\vee}({\bf M}),{\bf k}\rangle_I.
\end{align*}
Therefore, we have
\begin{align*}
\left(\te_i^*({\bf M}^{\sharp})\right)^{\sharp}_{\bf k}
&=\te_i^*({\bf M}^{\sharp})_{{\bf k}^c}-
\langle \mbox{wt}^{\vee}(\te_i^*({\bf M}^{\sharp})),{\bf k}\rangle_I\\
&={M}_{\bf k}+\langle \mbox{wt}^{\vee}({\bf M}),{\bf k}\rangle_I
-\langle \mbox{wt}^{\vee}({\bf M}),{\bf k}\rangle_I
-\langle h_i,{\bf k}\rangle_I
\qquad\mbox{by Lemma \ref{lemma:sharp1} (1)}\\
&=M_{\bf k}-\langle h_i,{\bf k}\rangle_I\\
&=\begin{cases}
M_{\bf k}+1 & \mbox{if }{\bf k}\in \cM_I^{\times}(i)^*,\\
M_{\bf k} & \mbox{if }\cM_I^{\times}\setminus
\left(\cM_I^{\times}(i)\cup\cM_I^{\times}(i)^*\right).
\end{cases}
\end{align*}
This proves the lemma.
\end{proof}

\begin{prop}\label{prop:te}
{\rm (1)} Assume that  $\eps_i({\bf M})>0$. Then, there exists a unique $e$-BZ
datum ${\bf M}^{[1]}$ such that 
\vskip 1mm
{\rm (a)} $({\bf M}^{[1]})_{\bf k}=M_{\bf k}+1$ for ${\bf k}\in\cM_I^{\times}(i)^*$,
\vskip 1mm
{\rm (b)} $({\bf M}^{[1]})_{\bf k}=M_{\bf k}$ for ${\bf k}\in\cM_I^{\times}
\setminus\left(\cM_I^{\times}(i)\cup\cM_I^{\times}(i)^*\right).$
\vskip 1mm
\noindent
{\rm (2)}  There exists a unique $e$-BZ datum ${\bf M}^{[2]}$ such that
\vskip 1mm
{\rm (a)} $({\bf M}^{[2]})_{\bf k}=M_{\bf k}-1$ for ${\bf k}\in\cM_I^{\times}(i)^*$,
\vskip 1mm
{\rm (b)} $({\bf M}^{[2]})_{\bf k}=M_{\bf k}$ for ${\bf k}\in\cM_I^{\times}
\setminus\left(\cM_I^{\times}(i)\cup\cM_I^{\times}(i)^*\right).$
\vskip 1mm
\noindent
\end{prop}
\begin{proof}
Since part (2) is proved in a similar way, we only give a proof of part (1). 
The existence of the required ${\bf M}^{[1]}$ is already
proved in Lemma \ref{lemma:sharp4}. 
Let ${\bf N}^{[1]}$ be another $e$-BZ  datum which satisfy conditions (a) and (b).
For the uniqueness, it suffices to show
that ${\bf M}^{[1]}_{{\bf k}}={\bf N}^{[1]}_{{\bf k}}$ for an arbitrary subinterval 
${\bf k}=[s+1,t]$ of $\widetilde{I}$, where $\widetilde{I}=[n+1,n+m+1]$.
If $[s+1,t]\in\cM_I^{\times}\setminus \cM_I^{\times}(i)$, then the assertion is
obvious from conditions (a) and (b). Assume that $[s+1,t]\in
\cM_I^{\times}(i)$. Here we note
that such an interval $[s+1,t]$ has the following form:
$$[s+1,i],\quad n\leq s\leq i-1.$$
If $s=n$, then we have $({\bf M}^{[1]})_{[n+1,i]}=({\bf N}^{[1]})_{[n+1,i]}=0$ by 
the normalization condition. 
Now, suppose that $({\bf M}^{[1]})_{[s,i]}=({\bf N}^{[1]})_{[s,i]}$. 
Then, by the tropical Pl\"ucker
relation for ${\bf k}=[s+1,i-1]$ and $s<i<i+1$, we have
\begin{align*}
&({\bf M}^{[1]})_{[s+1,i]}+({\bf M}^{[1]})_{[s,i-1]\cup\{i+1\}}\\
&\qquad\qquad
+\mbox{min}\left\{
({\bf M}^{[1]})_{[s,i-1]}+({\bf M}^{[1]})_{[s+1,i+1]},~
({\bf M}^{[1]})_{[s+1,i-1]\cup\{i+1\}}+({\bf M}^{[1]})_{[s,i]}
\right\}.
\end{align*}
Also, by conditions (a) and (b), we have
$$({\bf M}^{[1]})_{[s,i-1]\cup\{i+1\}}=M_{[s,i-1]\cup\{i+1\}}+1,\quad
({\bf M}^{[1]})_{[s+1,i-1]\cup\{i+1\}}=M_{[s+1,i-1]\cup\{i+1\}}+1,$$
$$({\bf M}^{[1]})_{[s,i-1]}=M_{[s,i-1]},\quad 
({\bf M}^{[1]})_{[s+1,i+1]}=M_{[s+1,i+1]}.$$
Therefore, we deduce that
\begin{align*}
({\bf M}^{[1]})_{[s+1,i]}&=-M_{[s,i-1]\cup\{i+1\}}-1\\
&\qquad
+\mbox{min}\left\{
M_{[s,i-1]}+M_{[s+1,i+1]},~
M_{[s+1,i-1]\cup\{i+1\}}+1+({\bf M}^{[1]})_{[s,i]}
\right\}.
\end{align*}
Similarly, we obtain
\begin{align*}
({\bf N}^{[1]})_{[s+1,i]}&=-M_{[s,i-1]\cup\{i+1\}}-1\\
&\qquad
+\mbox{min}\left\{
M_{[s,i-1]}+M_{[s+1,i+1]},~
M_{[s+1,i-1]\cup\{i+1\}}+1+({\bf N}^{[1]})_{[s,i]}
\right\}.
\end{align*}
Consequently, we obtain $({\bf M}^{[1]})_{[s+1,i]}=({\bf N}^{[1]})_{[s+1,i]}$. 
This proves the proposition.
\end{proof}
\begin{cor}
For ${\bf k}\in\cM_I^{\times}(i)$, we have
$$({\bf M}^{[2]})_{\bf k}=\mbox{\rm min}\left\{
M_{\bf k}+1,~M_{\sigma_i {\bf k}}+\eps_i({\bf M})
\right\}.$$
\end{cor}
\begin{proof}
From the uniqueness of ${\bf M}^{[2]}$, it follows that 
${\bf M}^{[2]}=\bigl(\tf_i^*({\bf M}^{\sharp})\bigr)^{\sharp}$.
Therefore,
\begin{align*}
({\bf M}^{[2]})_{\bf k}
&=\bigl(\tf_i^*({\bf M}^{\sharp})\bigr)^{\sharp}_{\bf k}\\
&=\bigl(\tf_i^*({\bf M}^{\sharp})\bigr)_{{\bf k}^c}
-\langle \mbox{wt}^{\vee}(\tf_i^*({\bf M}^{\sharp})),{\bf k}\rangle_I\\
&=\mbox{min}\left\{({\bf M}^{\sharp})_{{\bf k}^c},
({\bf M}^{\sharp})_{\sigma_i{\bf k}^c}+c_i^*({\bf M}^{\sharp})\right\}
-\langle \mbox{wt}^{\vee}({\bf M}),{\bf k}\rangle_I
+\langle h_i,{\bf k}\rangle_I\\
&=\mbox{min}\left\{M_{\bf k}+1,
M_{\sigma_i\bf k}+\langle \mbox{wt}^{\vee}({\bf M}),\sigma_i{\bf k}
-{\bf k}\rangle_I+c_i^*({\bf M}^{\sharp})+1\right\}.
\end{align*}
Here, we remark that $\langle h_i,{\bf k}\rangle=1$ since ${\bf k}\in 
\cM_I^{\times}(i)$.
Let us compute the second term on the right-hand side of the last equality.
Note that $\sigma_i{\bf k}-{\bf k}=-\langle h_i,{\bf k}\rangle_I\alpha_i^I
=-\alpha_i^I$. Hence we deduce that
\begin{align*}
\mbox{the second term}
&=M_{\sigma_i\bf k}-\langle \mbox{wt}^{\vee}({\bf M}),\alpha_i^I\rangle_I
+\langle h_i,\mbox{wt}({\bf M}^{\sharp})\rangle_I+\eps_i^*({\bf M}^{\sharp})-1
+1\\
&=M_{\sigma_i\bf k}+\eps_i({\bf M}).
\end{align*}
This proves the corollary.
\end{proof}
\subsection{Ordinary crystal structure on $\cBZ_I^e$}
We define another crystal structure on $\cBZ_I^e$ via the bijections
$\cB_I\overset{\sim}{\to}\bigsqcup_{\nu\in Q_+}\mbox{Irr}\Lambda(\nu)
\overset{\sim}{\to}\cBZ_I^e$.
Let ${\bf M}=(M_{\bf k})_{{\bf k}\in\cM_I^{\times}}$ be an $e$-BZ datum. 
Then, there exists a unique Lusztig datum ${\bf a}$
(or equivalently, a unique irreducible Lagrangian $\Lambda_{\bf a}$) such that 
${\bf M}={\bf M}({\bf a})$. Now we define 
$$\eps_i({\bf M}):=\eps_i({\bf a})=\eps_i(\Lambda_{\bf a}),\qquad
\vphi_i({\bf M}):=\vphi_i({\bf a})=\vphi_i(\Lambda_{\bf a}),$$
$$\te_i{\bf M}:=\begin{cases}
{\bf M}(\te_i{\bf a})={\bf M}(\te_i\Lambda_{\bf a}) & \mbox{if }~\eps_i({\bf a})>0,\\
0 & \mbox{if }~\eps_i({\bf a})=0,
\end{cases}
\quad\mbox{and}\quad
\tf_i{\bf M}:={\bf M}(\tf_i{\bf a})={\bf M}(\tf_i\Lambda_{\bf a}).$$
By the definitions, it is obvious that the set $\cBZ_I^e$, equipped with the maps
$\mbox{wt}$, $\eps_i$, $\vphi_i$, $\te_i$, $\tf_i$, is a $U_q(\gtsl_{m+1})$-crystal,
and the bijections above give rise to isomorphisms of crystals
$$\left(\cB_I;\mbox{\rm wt},\eps_i,\vphi_i,\te_i,\tf_i\right)
\overset{\sim}{\longrightarrow}
\left(\bigsqcup_{\nu\in Q_+}\mbox{Irr}\Lambda(\nu)
;\mbox{\rm wt},\eps_i,\vphi_i,\te_i,\tf_i\right)
\overset{\sim}{\longrightarrow}
\left(\cBZ_I^{e};\mbox{\rm wt},\eps_i,\vphi_i,\te_i,\tf_i\right).$$
We call this crystal structure the ordinary crystal structure on $\cBZ_I^e$. 

\begin{lemma}\label{lemma:ast=sharp}
For $\Lambda\in\mbox{\rm Irr}\Lambda(\nu)$, we have
$${\bf M}(\Lambda^*)={\bf M}(\Lambda)^{\sharp}.$$
\end{lemma}
\begin{proof}
We write $\nu=\sum_{i\in I}\nu_i\alpha_i^I$.
Let $B$ be a general point of $\Lambda$. Then its transpose ${}^tB$ is
also a general point of $\Lambda^*$. Therefore, we compute:
\begin{align*}
{M}_{\bf k}(\Lambda^*)&={M}_{\bf k}({}^tB)\\
&=-\dimc\mbox{Coker}
\left(\mathop{\bigoplus}_{k\in \mbox{\rm out}({\bf k})}V(\nu)_k
\overset{\oplus {}^tB_{\mu}}{\longrightarrow}
\mathop{\bigoplus}_{l\in \mbox{\rm in}({\bf k})}V(\nu)_l\right)\\
&=-\dimc\mbox{Ker}
\left(\mathop{\bigoplus}_{l\in \mbox{\rm in}({\bf k})}V(\nu)_l
\overset{\oplus {}B_{\mu}}{\longrightarrow}
\mathop{\bigoplus}_{k\in \mbox{\rm out}({\bf k})}V(\nu)_k\right)\\
&=-\dimc\mbox{Coker}
\left(\mathop{\bigoplus}_{l\in \mbox{\rm out}({\bf k}^c)}V(\nu)_l
\overset{\bigoplus {}B_{\mu}}{\longrightarrow}
\mathop{\bigoplus}_{k\in \mbox{\rm in}({\bf k}^c)}V(\nu)_k\right)\\
&\qquad\qquad\qquad\qquad
+\sum_{k\in \mbox{\rm in}({\bf k}^c)}\dimc V(\nu)_k
-\sum_{l\in\mbox{\rm out}({\bf k}^c)}\dimc V(\nu)_l\\
&={M}_{{\bf k}^c}(\Lambda)+\sum_{k\in \mbox{\rm out}({\bf k})}\nu_k
-\sum_{l\in\mbox{\rm in}({\bf k})}\nu_l.
\end{align*}  
Here, for the third equality, we take the transpose ${}^t({}^tB)=B$ of ${}^tB$. 
By the definitions of $\mbox{\rm out}({\bf k})$ and $\mbox{\rm in}({\bf k})$, 
we have
$$\langle h_p,{\bf k}\rangle_I=\begin{cases}
1 & \mbox{if $p\in \out({\bf k})$},\\
-1 & \mbox{if $p\in \inn({\bf k})$},\\
0 & \mbox{otherwise},
\end{cases}$$
and hence
$$\langle \mbox{wt}^{\vee}({\bf M}(\Lambda)),{\bf k}\rangle_I
=-\sum_{k\in \mbox{\rm out}({\bf k})}\nu_k
+\sum_{l\in\mbox{\rm in}({\bf k})}\nu_l.$$
From these, it follows that
$$M_{\bf k}(\Lambda)={M}_{{\bf k}^c}(\Lambda)
-\langle \mbox{wt}^{\vee}({\bf M}(\Lambda)),{\bf k}\rangle_I=
\bigl({\bf M}(\Lambda)^{\sharp}\bigr)_{\bf k}.$$
This proves the lemma.
\end{proof}

\begin{prop}\label{prop:sharp}
As operators on $\cBZ_I^e$, 
$$\te_i=\sharp\circ\te_i^*\circ\sharp\quad\mbox{and}\quad
\tf_i=\sharp\circ\tf_i^*\circ\sharp.$$
\end{prop}

\begin{proof}
Let ${\bf M}\in \cBZ_I^e$. We only give a proof of the first equality, since
the proof of the second one is similar. 

If $\eps_i({\bf M})=\eps_i^*({\bf M}^{\sharp})=0$, 
then $\te_i{\bf M}=\bigl(\te_i^*({\bf M}^{\sharp})\bigr)^{\sharp}=0$ by
the definitions.
So, assume that $\eps_i({\bf M})=\eps_i^*({\bf M}^{\sharp})>0$. 
Let $\Lambda$ be a unique irreducible Lagrangian such that 
${\bf M}={\bf M}(\Lambda)$. Since the bijection $\Xi_I:\Lambda \mapsto
{\bf M}(\Lambda)$ is an isomorphism with respect to both the ordinary and 
$\ast$-crystal structures, we have
\begin{align*}
&\bigl(\te_i^*({\bf M}^{\sharp})\bigr)^{\sharp}=
\bigl(\te_i^*({\bf M}(\Lambda)^{\sharp})\bigr)^{\sharp}=
\bigl(\te_i^*({\bf M}(\Lambda^*))\bigr)^{\sharp}=
{\bf M}(\te_i^*\Lambda^*)^{\sharp}=
{\bf M}\bigl((\te_i\Lambda)^*\bigr)^{\sharp}={\bf M}(\te_i\Lambda)\\
&\qquad=\te_i{\bf M}(\Lambda)=\te_i{\bf M},
\end{align*}
as desired.
\end{proof}
The following corollary is obvious by the consideration above.
\begin{cor}
{\rm (1)} Let ${\bf M}\in \cBZ_I^e$, and assume that $\eps_i({\bf M})>0$. 
Then, $\te_i{\bf M}$ is a unique 
$e$-BZ datum such that
\vskip 1mm
{\rm (a)} $(\te_i M)_{\bf k}=M_{\bf k}+1$ for ${\bf k}\in\cM_I^{\times}(i)^*$,
\vskip 1mm
{\rm (b)} $(\te_i M)_{\bf k}=M_{\bf k}$ for ${\bf k}\in\cM_I^{\times}\setminus
\left(\cM_I^{\times}(i)\cup\cM_I^{\times}(i)^*\right).$
\vskip 1mm
\noindent
{\rm (2)}  For every ${\bf M}\in\cBZ_I^e$,
$$(\tf_i{\bf M})_{\bf k}=\begin{cases}
\mbox{\rm min}\left\{M_{\bf k}+1,M_{\sigma_i{\bf k}}+\eps_i({\bf M})\right\}
&\mbox{if }~{\bf k}\in \cM_I^{\times}(i).\\
M_{\bf k}-1&\mbox{if }~{\bf k}\in\cM_I^{\times}(i)^*,\\
M_{\bf k} &\mbox{if }~{\bf k}\in\cM_I^{\times}\setminus
\left(\cM_I^{\times}(i)\cup\cM_I^{\times}(i)^*\right).
\end{cases}$$
\end{cor}

\begin{prop}\label{prop:fin-strict}
Let $i,j\in I$, and ${\bf M}\in \cBZ_{I}^e$.
Set $c:=\eps_i({\bf M})$ and ${\bf M}'=\te_p^c{\bf M}$. 
\vskip 1mm
\noindent
{\rm (1)} We have
$$\eps^*_i({\bf M})=\mbox{\rm max}\left\{\eps^*_i({\bf M}'),~c-
\bigl\langle {h}_i,\mbox{\rm wt}({\bf M}')\bigr\rangle_I\right\}.$$
{\rm (2)} If $i\ne j$ and $\eps_j^*({\bf M})>0$, then
$$\eps_j(\te_p^*{\bf M})=c,\qquad \te_j^c(\te_i^*{\bf M})=\te_i^*{\bf M}'.$$
{\rm (3)} If $\eps_i^*({\bf M})>0$, then we have
$$\eps_i(\te_i^*{\bf M})=\begin{cases}
\eps_i({\bf M}) & \mbox{if}\quad\eps^*_i({\bf M}')\geq 
c-\bigl\langle {h}_i,\mbox{\rm wt}({\bf M}')\bigr\rangle_I,\\
\eps_i({\bf M})-1& \mbox{if}\quad\eps^*_i({\bf M}')<
c-\bigl\langle {h}_i,\mbox{\rm wt}({\bf M}')\bigr\rangle_I,
\end{cases}$$
and
$$\te_i^{c'}\bigl(\te_i^*{\bf M}\bigr)=
\begin{cases}
\te_i^*{\bf M}' & \mbox{if}\quad\eps^*_i({\bf M}')\geq 
c-\bigl\langle {h}_i,\mbox{\rm wt}({\bf M}')\bigr\rangle_I,\\
{\bf M}' & \mbox{if}\quad\eps^*_i({\bf M}')<
c-\bigl\langle {h}_i,\mbox{\rm wt}({\bf M}')\bigr\rangle_I.
\end{cases}$$
Here, we set $c':=\eps_i(\te_i^*{\bf M}).$
\end{prop}
\begin{proof}
Recall that the bijection $\Xi_I:\Lambda \mapsto
{\bf M}(\Lambda)$ is an isomorphism with respect to both the ordinary and 
$\ast$-crystal structures. Therefore, all of the desired equations follow 
immediately from the corresponding ones, which hold in 
$\bigsqcup_{\nu\in Q_+}\mbox{Irr}\Lambda(\nu)$ (see \cite{KS}). 
This proves the proposition.
\end{proof}
\section{Ordinary crystal structure on $(\cBZ_{\nz}^e)^{\sigma}$}
\subsection{Definition of ordinary Kashiwara operators on $\cBZ_{\nz}^e$}
For ${\bf M}=(M_{\bf k})_{{\bf k}\in\cM_{\nz}}\in \cBZ_{\nz}^e$ and $p\in\nz$,
we set
$$\eps_p({\bf M}):=-M_{{\bf k}(\sigma_p\Lambda_p)}.$$
Observe that if ${\bf k}(\sigma_p\Lambda_p)\in \cM_{\nz}(I)$, then
$$\eps_p({\bf M})=-M_{{\bf k}(\sigma_p\Lambda_p)}=
-({\bf M}_I)_{\mbox{res}_I({\bf k}(\sigma_p\Lambda_p))}
=-({\bf M}_I)_{{\bf k}(\sigma_p\varpi_p^I)}=\eps_p({\bf M}_I).$$
First, let us define the ordinary raising Kashiwara operators on  $\cBZ_{\nz}^e$. 
If $\eps_p({\bf M})>0$, the we define a new collection ${\bf M}^{[1]}=
\left(M_{\bf k}^{[1]}\right)_{{\bf k}\in\cM_{\nz}}$ of integers as follows. 
For a given ${\bf k}\in \cM_{\nz}$,  take a finite interval $I$ in $\nz$
such that ${\bf k},\sigma_p{\bf k}\in \cM_{\nz}(I)$, and 
$I\in\mbox{Int}^e({\bf M};{\bf k}(\Lambda_p))\cap
\mbox{Int}^e({\bf M};{\bf k}(\sigma_p\Lambda_p))\cap
\mbox{Int}^e({\bf M};{\bf k}(\Lambda_{p+1}))\cap
\mbox{Int}^e({\bf M};{\bf k}(\Lambda_{p-1})).$ Then, we set
$$M_{\bf k}^{[1]}:=\left(\te_p{\bf M}_I\right)_{\mbox{res}_I({\bf k})}.$$
Here, $\te_p$ is the ordinary raising Kashiwara operator on $\cBZ_I^e$ defined
in the previous section. 
Now, we define the action of $\te_p$ on $\cBZ_{\nz}^e$ by
$$\te_p{\bf M}:=\begin{cases}
{\bf M}^{[1]} & \mbox{if }\eps_p({\bf M})>0,\\
0 & \mbox{if }\eps_p({\bf M})=0.
\end{cases}$$
Note that the definition above does not depend on the choice of $I$.

Next, let us define the ordinary lowering Kashiwara operators on $\cBZ_{\nz}^e$.
For ${\bf M}=(M_{\bf k})_{{\bf k}\in\cM_{\nz}}\in \cBZ_{\nz}^e$ and $p\in\nz$, 
we define a new collection $\tf_p{\bf M}=(M_{\bf k}^{[2]})_{{\bf k}\in\cM_{\nz}}$
of integers as follows. For a given ${\bf k}\in \cM_{\nz}$, take a finite interval $I$
in $\nz$ such that ${\bf k}, \sigma_p{\bf k}\in
\cM_{\nz}(I)$, and $I\in \mbox{Int}^e({\bf M};{\bf k}(\Lambda_p))\cap 
\mbox{Int}^e({\bf M};{\bf k}(\sigma_p\Lambda_p))$. Then we set
$$M_{\bf k}^{[2]}:=\left(\tf_p{\bf M}_I\right)_{\mbox{res}_I({\bf k})}.$$
Here, $\tf_p$ is the ordinary lowering Kashiwara operator on $\cBZ_I^e$ defined
in the previous section. \\

For $p\in\nz$, we set
$$\cM_{\nz}(p):=\{{\bf k}\in\cM_{\nz}~|~p\in {\bf k},~p+1\not
\in{\bf k}\},\quad
\cM_{\nz}(p)^*:=\{{\bf k}\in\cM_{\nz}~|~p\not\in {\bf k},~p+1
\in{\bf k}\}.
$$
The following lemma follows easily from the definitions.
\begin{lemma}\label{cor:ord-kas}
Let ${\bf M}=(M_{\bf k})_{{\bf k}\in\cM_{\nz}}\in \cBZ_{\nz}^e$.
\vskip 1mm
\noindent
{\rm (1)} If $\eps_p({\bf M})>0$, then
\vskip 1mm
{\rm (a)} $(\te_p M)_{\bf k}=M_{\bf k}+1$ for ${\bf k}\in\cM_{\nz}^{\times}(p)^*$,
\vskip 1mm
{\rm (b)} $(\te_p M)_{\bf k}=M_{\bf k}$ for ${\bf k}\in\cM_{\nz}^{\times}\setminus
\left(\cM_{\nz}^{\times}(p)\cup\cM_{\nz}^{\times}(p)^*\right).$
\vskip 1mm
\noindent
{\rm (2)}  For each ${\bf M}\in\cBZ_{\nz}^e$, we have
$$(\tf_p{\bf M})_{\bf k}=\begin{cases}
\mbox{\rm min}\left\{M_{\bf k}+1,M_{\sigma_p{\bf k}}+\eps_p({\bf M})\right\}
&\mbox{if }~{\bf k}\in \cM_{\nz}^{\times}(p),\\
M_{\bf k}-1&\mbox{if }~{\bf k}\in\cM_{\nz}^{\times}(p)^*,\\
M_{\bf k} &\mbox{if }~{\bf k}\in\cM_{\nz}^{\times}\setminus
\left(\cM_{\nz}^{\times}(p)\cup\cM_{\nz}^{\times}(p)^*\right).
\end{cases}$$
\end{lemma}

\begin{prop}\label{prop:well-def}
{\rm (1)} If $\eps_p({\bf M})>0$, then $\te_p{\bf M}\in \cBZ_{\nz}^e$.
\vskip 1mm
\noindent
{\rm (2)} For every ${\bf M}\in\cBZ_{\nz}^e$ and $p\in\nz$, we have
$\tf_p{\bf M}\in\cBZ_{\nz}^e$. 
\end{prop}
In the next subsection, we give a proof of this proposition.
\subsection{Proof of Proposition \ref{prop:well-def}}
Since part (2) is obtained in a similar way, we only give a proof of part (1). 
We will only verify that condition 
(2-b) in Definition \ref{defn:BZ} is satisfied for $\te_p{\bf M}$ with
$\eps_p({\bf M})>0$, since the remaining ones 
are easily verified. 
Namely, we will prove the following: 
\vskip 3mm
\noindent
{\bf Claim 1}. {\it Assume that $\eps_p({\bf M})>0$, and let ${\bf k}\in\cM_{\nz}$. 
Then, there exists a finite interval 
$I$ in $\nz$ such that for every $J\supset I$,  
$$\left(\te_p{\bf M}\right)_{\Omega_I({\bf k})}=
\left(\te_p{\bf M}\right)_{\Omega_J({\bf k})}.\eqno{(4.2.1)}$$}

Take a finite interval $K=[n_K+1,n_K+m_K]$ in $\nz$ such that 
${\bf k}, \sigma_p{\bf k}\in\cM_{\nz}(K)$,
and $K\in \mbox{Int}^e({\bf M};{\bf k}(\Lambda_p))\cap
\mbox{Int}^e({\bf M};{\bf k}(\sigma_p\Lambda_p))\cap
\mbox{Int}^e({\bf M};{\bf k}(\Lambda_{p+1}))\cap
\mbox{Int}^e({\bf M};{\bf k}(\Lambda_{p-1}))$. 
Set $K':=[n_K,n_k+m_K+1]$. Since $\cM_{\nz}(K')$ is a finite set, 
we can take a finite interval $I=[n_I+1,n_I+m_I]$ in $\nz$ with the 
following properties:
$$I\supset K,\quad\mbox{and}\quad I\in
\mbox{Int}^e({\bf M};{\bf n})
\mbox{ for all }{\bf n}\in \cM_{\nz}({K'}).\eqno{(4.2.2)}$$

In the following, we will show that such an interval $I$ satisfies the condition in 
Claim 1. We may assume that $J=\{n_I\}\cup I$ (case (i)) or 
$J=I\cup \{n_I+m_I+1\}$ (case (ii)).
We show equation (4.2.1) only in case (i); the assertion in case (ii) follows by
a similar (and easier) argument.\\

Before starting a proof, we give some lemmas. We set
${\bf a}=(a_{i,j})_{(i,j)\in\Delta_I^+}:=\Psi_I^{-1}({\bf M}_I)\in\cB_I$,
${\bf b}=(b_{k,l})_{(k,l)\in\Delta_J^+}:=\Psi_J^{-1}({\bf M}_J)\in\cB_J$,
$\Lambda_{\bf a}:=\Xi_I^{-1}({\bf M}_I)\in\mbox{Irr}\Lambda(\nu_I)$,
and
$\Lambda_{\bf b}:=\Xi_J^{-1}({\bf M}_J)\in\mbox{Irr}\Lambda(\nu_J)$, 
where $\nu_I=\mbox{wt}({\bf M}_I)$ and $\nu_J=\mbox{wt}({\bf M}_J)$.
Let $B^I=(B_{\tau}^I)\in \Lambda_{\bf a}$ and 
$B^J=(B_{\tau}^J)\in \Lambda_{\bf b}$ be general points. 

\begin{lemma}\label{lemma:claim2}
Let $\sigma(n_I\to n_K)$ be the path from $n_I$ to $n_K$ defined as follows:
\begin{center}
\setlength{\unitlength}{1mm}
\begin{picture}(83,10)
\put(-10,5){$\sigma(n_I\to n_K):$}
\put(17,6){\vector(1,0){6}}\put(17,6){\line(1,0){10}}
\put(29,6){\vector(1,0){6}}\put(29,6){\line(1,0){10}}
\put(41,6){\vector(1,0){6}}\put(41,6){\line(1,0){10}}
\put(54,4.5){$\cdots$}
\put(64,6){\vector(1,0){6}}\put(64,6){\line(1,0){10}}
\put(76,6){\vector(1,0){6}}\put(76,6){\line(1,0){10}}
\put(88,6){\vector(1,0){6}}\put(88,6){\line(1,0){10}}
\put(15,1){{\scriptsize $n_I$}}
\put(23.5,1){{\scriptsize $n_I+1$}}
\put(35.5,1){{\scriptsize $n_I+2$}}
\put(70.5,1){{\scriptsize $n_K-2$}}
\put(82.5,1){{\scriptsize $n_K-1$}}
\put(97.5,1){{\scriptsize $n_K$}}
\put(16,6){\circle{2}}\put(28,6){\circle{2}}
\put(40,6){\circle{2}}\put(75,6){\circle{2}}
\put(87,6){\circle{2}}\put(99,6){\circle{2}}
\put(101,5){.}
\end{picture}
\end{center}
Then, the corresponding composite map
$B^J_{\sigma(n_I\to n_K)}:V(\nu_J)_{n_I}\to V(\nu_J)_{n_K}$ is a zero map.
\end{lemma}
\begin{proof}
Since $\Lambda_{n_K}\in \cM_{\nz}(K')$, we have
\begin{align*}
&({\bf M}_I)_{[n_K+1,n_I+m_I+1]}=M_{\Omega_I(\Lambda_{n_K})}\\
&\qquad\qquad\qquad
=M_{\Omega_J(\Lambda_{n_K})}=({\bf M}_J)_{[n_K+1,n_I+m_I+1]}=-(\nu_J)_{n_K}
.\end{align*}
Also, by the definition, we have 
\begin{align*}
({\bf M}_I)_{[n_K+1,n_I+m_I+1]}&=({\bf M}_J)_{\{n_I\}\cup[n_K+1,n_I+m_I+1]}\\
&=-\dimc\mbox{Coker}\left(V(\nu_J)_{n_I}\overset{B^J}{\longrightarrow}
V(\nu_J)_{n_K}\right).
\end{align*}
From these, we obtain
$$(\nu_J)_{n_K}=\dimc\mbox{Coker}\left(V(\nu_J)_{n_I}
\overset{B^J}{\longrightarrow}V(\nu_J)_{n_K}\right).$$
This shows that the map $B^J_{\sigma(n_I\to n_K)}$ is a zero map, as desired .
\end{proof}

Let ${\bf a}^*=(a^*_{i,j})_{(i,j)\in\Delta_I^+}:=\Psi_I^{-1}({\bf M}_I^{\sharp})$ and
${\bf b}^*=(b^*_{k,l})_{(k,l)\in\Delta_J^+}:=\Psi_J^{-1}({\bf M}_J^{\sharp})$. 

\begin{lemma}\label{lemma:claim5}
We have $b^*_{n_I,l}=0$ for $n_K+1\leq l\leq n_I+m_I+1$.
\end{lemma}
\begin{proof}
First, note that
$$({\bf M}_J^{\sharp})_{[n_I+1,n_I+m_I+1]}=-\sum_{l=n_I+1}^{n_I+m_I+1}b^*_{n_I,l}\quad
\mbox{and}\quad
({\bf M}_J^{\sharp})_{[n_I+1,n_K]}=-\sum_{l=n_I+1}^{n_K}b^*_{n_I,l}.$$
Next, by Lemma \ref{lemma:ast=sharp}, we have
\begin{align*}
({\bf M}_J^{\sharp})_{[n_I+1,n_K]}=M_{[n_I+1,n_K]}(\Lambda_{\bf b}^*)=
-\dimc\mbox{Coker}\left(V(\nu_J)_{n_K}\overset{{}^t(B_J)}{\longrightarrow}
V(\nu_J)_{n_I}\right).
\end{align*}
Since $\bigl({}^t(B^J)\bigr)_{\sigma(n_K\to n_I)}
={}^t\left(B^J_{\sigma(n_I\to n_K)}\right)=0$ by Lemma \ref{lemma:claim2}, 
we deduce that  
$$({\bf M}_J^{\sharp})_{[n_I+1,n_K]}=-(\nu_J)_{n_I}.$$
Therefore, we deduce that
\begin{align*}
\sum_{l=n_K+1}^{n_I+m_I+1}b^*_{n_I,l}
&=-({\bf M}_J^{\sharp})_{[n_I+1,n_I+m_I+1]}+({\bf M}_J^{\sharp})_{[n_I+1,n_K]}
=0.
\end{align*}
Since $b^*_{n_I,l}$ is nonnegative for all $l$, we obtain the desired equality.
\end{proof}

\begin{lemma}\label{lemma:claim4}
For every $n_I+1\leq s\leq t\leq n_I+m_I+1$, we have 
$$\bigl({\bf M}_J^{\sharp}\bigr)_{[s,t]}=\bigl({\bf M}_I^{\sharp}\bigr)_{[s,t]}
+(\nu_I)_{s-1}-(\nu_J)_{s-1}-(\nu_I)_t+(\nu_J)_t.$$
Here, by convention,  $(\nu_I)_{s-1}=0$ for $s=n_I+1$ and $(\nu_I)_t=
(\nu_J)_t=0$ for $t=n_I+m_I+1$.
\end{lemma}
\begin{proof}
We assume that $n_I+1<s\leq t< n_I+m_I+1$; in the remaining case,
the desired equation follows by a similar (and easier) argument.\\

Write ${\bf n}_I=[s,t]\in \cM_I^{\times}$. Then we have
\begin{align*}
\bigl({\bf M}_I^{\sharp}\bigr)_{{\bf n}_I}&=
-\dimc\mbox{Coker}\left(V(\nu_I)_t\overset{{}^t(B_I)}{\longrightarrow}
V(\nu_I)_{s-1}\right)\\
&=-\dimc\mbox{Ker}\left(V(\nu_I)_{s-1}\overset{B_I}{\longrightarrow}
V(\nu_I)_t\right)\\
&=-\dimc\mbox{Coker}\left(V(\nu_I)_{s-}\overset{B_I}{\longrightarrow}
V(\nu_I)_t\right)-(\nu_I)_{s-1}+(\nu_I)_t\\
&=({\bf M}_I)_{{\bf n}_I^c}-(\nu_I)_{s-1}+(\nu_I)_t.
\end{align*}
Similarly, we have
$$\bigl({\bf M}_J^{\sharp}\bigr)_{{\bf n}_J}=({\bf M}_J)_{{\bf n}_J^c}
-(\nu_J)_{s-1}+(\nu_J)_t.$$
Here we set ${\bf n}_J:=[s,t]\in\cM_J^{\times}$. Since ${\bf n}_J^c=\{n_I\}\cup
{\bf n}_I^c$, we obtain
$$({\bf M}_I)_{{\bf n}_I^c}=({\bf M}_I)_{{\bf n}_J^c}.$$
Therefore, we deduce that
\begin{align*}
\bigl({\bf M}_J^{\sharp}\bigr)_{{\bf n}_J}&=({\bf M}_I)_{{\bf n}_I^c}
-(\nu_J)_{s-1}+(\nu_J)_t\\
&=\bigl({\bf M}_I^{\sharp}\bigr)_{{\bf n}_I}
+(\nu_I)_{s-1}-(\nu_I)_t-(\nu_J)_{s-1}+(\nu_J)_t.
\end{align*}
This proves the lemma.
\end{proof}

\begin{cor}\label{cor:claim6}
For every $n_I+1\leq i<j\leq n_I+m_I+1$, we have
$$b^*_{i,j}=a^*_{i,j}.$$ 
\end{cor}
\begin{proof} The desired equality follows easily from Lemma \ref{lemma:claim4}
and the chamber ansatz maps (see \cite{BFZ}):
$$b^*_{i,j}=\bigl({\bf M}_J^{\sharp}\bigr)_{[i,j]}+
\bigl({\bf M}_J^{\sharp}\bigr)_{[i+1,j-1]}-\bigl({\bf M}_J^{\sharp}\bigr)_{[i+,j]}
-\bigl({\bf M}_J^{\sharp}\bigr)_{[i,j-1]},$$
$$a^*_{i,j}=\bigl({\bf M}_I^{\sharp}\bigr)_{[i,j]}+
\bigl({\bf M}_I^{\sharp}\bigr)_{[i+1,j-1]}-\bigl({\bf M}_I^{\sharp}\bigr)_{[i+,j]}
-\bigl({\bf M}_I^{\sharp}\bigr)_{[i,j-1]}.$$
\end{proof}

\begin{prop}\label{prop:K-op}
We have
$$\left(\bigl(\te_p^*({\bf M}_I^{\sharp})\bigr)_K\right)_{[p+1,n_K+m_K+1]}=
\left(\bigl(\te_p^*({\bf M}_J^{\sharp})\bigr)_K\right)_{[p+1,n_K+m_K+1]}.$$
\end{prop}

\begin{proof}
Note that the desired equation is equivalent to the following:
$$\bigl(\te_p^*({\bf M}_I^{\sharp})\bigr)_{[n_I+1,n_K]\cup[p+1,n_K+m_K+1]}=
\bigl(\te_p^*({\bf M}_J^{\sharp})\bigr)_{[n_I,n_K]\cup[p+1,n_K+m_K+1]}.
\eqno{(4.2.3)}$$
Let ${\bf a}'=(a'_{i,j})_{(i,j)\in\Delta_I^+}:=\Psi_I^{-1}
\left(\te_p^*({\bf M}_I^{\sharp})\right)$ and
${\bf b}'=(b'_{k,l})_{(k,l)\in\Delta_J^+}:=\Psi_J^{-1}
\left(\te_p^*({\bf M}_J^{\sharp})\right)$, and
set ${\bf l}:=[n_I+1,n_K]\cup[p+1,n_K+m_K+1]$ and ${\bf m}:=[n_I,n_K]\cup
[p+1,n_K+m_K+1]$. 
Then, equation (4.2.3) is equivalent to the following:
$$M_{\bf l}({\bf a}')=M_{\bf m}({\bf b}').\eqno{(4.2.4)}$$

Observe that by Lemma \ref{lemma:claim5}, Corollary \ref{cor:claim6}, 
and the definition of the action of $\te_p^*$, 
$$a'_{i,j}=b'_{i,j}~((i,j)\in\Delta_I^+)\quad\mbox{and}\quad
b'_{n_I,l}=0~(n_K+1\leq l\leq n_I+m_I+1).\eqno{(4.2.5)}$$
Since the Maya diagrams ${\bf l}$ and ${\bf m}$ satisfy the condition of 
Lemma \ref{lemma:BZ-int} (1) with $s=n_K$, we have
\begin{align*}
M_{\bf l}({\bf a}')&=-\sum_{j=p+1}^{n_K+m_K+1}\sum_{i=n_I+1}^{l-1}a'_{i,j}\\
&\qquad\qquad+\mbox{\rm min}\left\{\left.
\sum_{t=n_K+1}^{2n_K+m_K-p+1}\sum_{s=n_I+1}^{t-1}a_{c_{s,t},c_{s,t}+(t-s)}'
~\right|\begin{array}{c}C=(c_{s,t})~\mbox{\rm is }\\ 
\mbox{\rm an ${\bf l}$-tableau}\end{array}\right\},
\end{align*}
and
\begin{align*}
M_{\bf m}({\bf b}')&=-\sum_{l=p+1}^{n_K+m_K+1}\sum_{k=n_I}^{l-1}b'_{k,l}\\
&\qquad\qquad+\mbox{\rm min}\left\{\left.
\sum_{t=n_K+1}^{2n_K+m_K-p+1}\sum_{s=n_I}^{t-1}b'_{d_{s,t},d_{s,t}+(t-s)}
~\right|\begin{array}{c}D=(d_{s,t})~\mbox{\rm is }\\ 
\mbox{\rm an ${\bf m}$-tableau}\end{array}\right\}.
\end{align*}
Here, by (4.2.5), 
$$\sum_{j=p+1}^{n_K+m_K+1}\sum_{i=n_I+1}^{l-1}a'_{i,j}=
\sum_{l=p+1}^{n_K+m_K+1}\sum_{k=n_I}^{l-1}b'_{k,l}.$$ 
Now, let $D=(d_{s,t})_{n_I\leq s\leq t\leq 2n_k+m_k-p+1}$ be an ${\bf m}$-tableau, 
and define two upper-triangular matrices 
$D'=(d'_{s,t})_{n_I\leq s\leq t\leq 2n_k+m_k-p+1}$ and
$D''=(d_{s,t}'')_{n_I+1\leq s\leq t\leq 2n_k+m_k-p+1}$ by 
$$d'_{s,t}:=\begin{cases}
n_I & \mbox{if }s=n_I,\\
d_{p,q} & \mbox{otherwise},
\end{cases}
\quad\mbox{and}\quad d''_{s,t}:=d_{s,t}.
$$
Then, $D'$ is an ${\bf m}$-tableau and $D''$ is an ${\bf l}$-tableau. 
From the definitions and (4.2.5), we see that
\begin{align*}
\sum_{t=n_K+1}^{2n_K+m_K-p+1}\sum_{s=n_I}^{t-1}b'_{d_{s,t},d_{s,t}+(t-s)}
&\geq \sum_{t=n_K+1}^{2n_K+m_K-p+1}\sum_{s=n_I}^{t-1}b'_{d'_{s,t},d'_{s,t}+(t-s)}\\
&=\sum_{t=n_K+1}^{2n_K+m_K-p+1}
\sum_{s=n_I+1}^{t-1}b'_{d''_{s,t},d''_{s,t}+(t-s)}\\
&=\sum_{t=n_K+1}^{2n_K+m_K-p+1}
\sum_{s=n_I+1}^{t-1}a'_{d''_{s,t},d''_{s,t}+(t-s)}.
\end{align*}
Therefore, we obtain
\begin{align*}
&\mbox{\rm min}\left\{\left.
\sum_{t=n_K+1}^{2n_K+m_K-p+1}\sum_{s=n_I+1}^{t-1}a_{c_{s,t},c_{s,t}+(t-s)}'
~\right|\begin{array}{c}C=(c_{s,t})~\mbox{\rm is }\\ 
\mbox{\rm an ${\bf l}$-tableau}\end{array}\right\}\\
&\qquad\qquad=
\mbox{\rm min}\left\{\left.
\sum_{t=n_K+1}^{2n_K+m_K-p+1}\sum_{s=n_I}^{t-1}b'_{d_{s,t},d_{s,t}+(t-s)}
~\right|\begin{array}{c}D=(d_{s,t})~\mbox{\rm is }\\ 
\mbox{\rm an ${\bf m}$-tableau}\end{array}\right\},
\end{align*}
and hence (4.2.4).
\end{proof}

Now let us start the proof of equation (4.2.1). By (4.2.2) and 
Proposition \ref{prop:sharp}, we have
\begin{align*}
\left(\te_p{\bf M}\right)_{\Omega_I({\bf k})}
&=\left(\te_p{\bf M}_I\right)_{\mbox{res}_I(\Omega_I({\bf k}))}
=\left(\te_p{\bf M}_I\right)_{{\bf k}_I^c}
=\left(\bigl(\te_p^*({\bf M}_I^{\sharp})\bigr)^{\sharp}\right)_{{\bf k}_I^c}\\
&=\bigl(\te_p^*({\bf M}_I^{\sharp})\bigr)_{{\bf k}_I}
-\langle \mbox{wt}^{\vee}(\te_p^*({\bf M}_I^{\sharp})),{\bf k}_I^c\rangle_I\\
&=\bigl(\te_p^*({\bf M}_I^{\sharp})\bigr)_{{\bf k}_I}
+\langle \mbox{wt}^{\vee}({\bf M}_I),{\bf k}_I\rangle_I
+\langle h_p,{\bf k}_I\rangle_I.
\end{align*}
Here, ${\bf k}_I:=\mbox{res}_I({\bf k})$ and ${\bf k}_I^c:=
\widetilde{I}\setminus {\bf k}_I$. Similarly, we have
$$\left(\te_p{\bf M}\right)_{\Omega_J({\bf k})}=
\bigl(\te_p^*({\bf M}_J^{\sharp})\bigr)_{{\bf k}_J}
+\langle \mbox{wt}^{\vee}({\bf M}_J),{\bf k}_J\rangle_J
+\langle h_p,{\bf k}_J\rangle_J.
$$
By these, the proof of equation (4.2.1) is reduced to showing:
\vskip 1mm 
(a) $\langle h_q,{\bf k}_I\rangle_I=\langle h_q,{\bf k}_J\rangle_J$ 
for all $q\in {K'}$;
\vskip 1mm 
(b) $\langle \mbox{wt}^{\vee}({\bf M}_I),{\bf k}_I\rangle_I=
\langle \mbox{wt}^{\vee}({\bf M}_J),{\bf k}_J\rangle_J$;
\vskip 1mm
(c) $\bigl(\te_p^*({\bf M}_I^{\sharp})\bigr)_{{\bf k}_I}=
\bigl(\te_p^*({\bf M}_J^{\sharp})\bigr)_{{\bf k}_J}$.\\

By the definitions, (a) is easily shown. 
Let us show (b).
Since ${\bf k}\in\cM_{\nz}(K')$, we have  
$\langle h_q,{\bf k}_I\rangle_I=0$ for $q\not\in {K'}$. 
Therefore, we see that   
\begin{align*}
\langle \mbox{wt}^{\vee}({\bf M}_I),{\bf k}_I\rangle_I&=
\sum_{q\in I}\langle h_q,{\bf k}_I\rangle_I ({\bf M}_I)_{[q+1,n_I+m_I+1]}
=\sum_{q\in I}\langle h_q,{\bf k}_I\rangle_I M_{\Omega_I({\bf k}(\Lambda_q))}
\\
&=\sum_{q\in K'}\langle h_q,{\bf k}_I\rangle_I M_{\Omega_I({\bf k}(\Lambda_q))}.
\end{align*}
Similarly, we see that
$$\langle \mbox{wt}^{\vee}({\bf M}_J),{\bf k}_J\rangle_J
=\sum_{q\in K'}\langle h_q,{\bf k}_J\rangle_J 
M_{\Omega_J({\bf k}(\Lambda_p))}.$$
Consequently, in view of (a), it suffices to show that 
$M_{\Omega_I({\bf k}(\Lambda_q))}=M_{\Omega_J({\bf k}(\Lambda_q))}$
for all $q\in {K'}$, which follows from (4.2.1). Thus, we have shown (b). For (c), 
it suffices to show the following proposition.

\begin{prop}\label{prop:e-ast-action}
$$\bigl(\te_p^*({\bf M}_I^{\sharp})\bigr)_K=
\bigl(\te_p^*({\bf M}_J^{\sharp})\bigr)_K.$$
\end{prop}
\begin{proof} 
We remark that 
$\bigl(\te_p^*({\bf M}_I^{\sharp})\bigr)_K$ and 
$\bigl(\te_p^*({\bf M}_J^{\sharp})\bigr)_K$ are both elements of $\cBZ_K^e$.
Hence it suffices to show the following:
\vskip 1mm
(d) $\left(\bigl(\te_p^*({\bf M}_I^{\sharp})\bigr)_K\right)_{[p+1,n_K+m_K+1]}
=\left(\bigl(\te_p^*({\bf M}_J^{\sharp})\bigr)_K\right)_{[p+1,n_K+m_K+1]}$;
\vskip 1mm
(e) $\left(\bigl(\te_p^*({\bf M}_I^{\sharp})\bigr)_K\right)_{\bf m}
=\left(\bigl(\te_p^*({\bf M}_J^{\sharp})\bigr)_K\right)_{\bf m}$
for all ${\bf m}\in \cM_K^{\times}\setminus \cM_K^{\times}(p)^*$.\\

Because (d) is already shown in Proposition \ref{prop:K-op},
the remaining task is to show (e). 
We set ${\bf m}^I:=(\mbox{res}_K^I)^{-1}({\bf m})$. Since
${\bf m}^I\in \cM_I^{\times}\setminus \cM_I^{\times}(p)^*$, we have
\begin{align*}
\left(\bigl(\te_p^*({\bf M}_I^{\sharp})\bigr)_K\right)_{\bf m}
&=\bigl(\te_p^*({\bf M}_I^{\sharp})\bigr)_{{\bf m}^I}
=({\bf M}_I^{\sharp})_{{\bf m}^I}=({\bf M}_I)_{({\bf m}^I)^c}
-\langle \mbox{wt}^{\vee}({\bf M}_I),{\bf m}^I\rangle_I.
\end{align*}
We set ${\bf m}^{\nz}:=\mbox{res}_K^{-1}({\bf m})\in \cM_{\nz}(K)\subset 
\cM_{\nz}$.
Then we have ${\bf m}^I=\mbox{res}_I({\bf m}^{\nz})$ and 
$({\bf m}^I)^c=\mbox{res}_I\bigl(\Omega_I({\bf m}^{\nz})\bigr)$. Therefore, 
we obtain
$$\left(\bigl(\te_p^*({\bf M}_I^{\sharp})\bigr)_K\right)_{\bf m}
=M_{\Omega_I({\bf m}^{\nz})}-
\langle \mbox{wt}^{\vee}({\bf M}_I),\mbox{res}_I({\bf m}^{\nz})\rangle_I.$$
Also, we obtain a similar equation, with $I$ replaced by $J$.
By the same argument as in the proof of (b), we deduce that 
$\langle \mbox{wt}^{\vee}({\bf M}_I),\mbox{res}_I({\bf m}^{\nz})\rangle_I
=\langle \mbox{wt}^{\vee}({\bf M}_J),\mbox{res}_J({\bf m}^{\nz})\rangle_J$.
Now it remains to verify that $M_{\Omega_I({\bf m}^{\nz})}=
M_{\Omega_J({\bf m}^{\nz})}$, which follows easily from (4.2.1).
Thus, we have shown (e). This proves the proposition.
\end{proof}
\subsection{Ordinary crystal structure on $(\cBZ_{\nz}^e)^{\sigma}$}
First, we give some properties of ordinary Kashiwara operators
on $\cBZ_{\nz}^e$. Because all of those are obtained by the same 
argument as in \cite{NSS}, we omit the proofs of them.
\begin{lemma}\label{lemma:propaties}
{\rm (1)} Let ${\bf M}\in \cBZ_{\nz}^e$ and $p\in\nz$. Then, $\te_p\tf_p{\bf M}
={\bf M}$. Also, if $\eps_p({\bf M})\ne 0$, then $\tf_p\te_p{\bf M}
={\bf M}$.
\vskip 1mm
\noindent
{\rm (2)} For ${\bf M}\in \cBZ_{\nz}^e$ and $p,q\in\nz$ with $|p-q|\geq 2$, 
we have
$\eps_p(\tf_p{\bf M})=\eps_p({\bf M})+1$ and 
$\eps_q(\tf_p{\bf M})=\eps_q({\bf M})$. Also, if $\eps_p({\bf M})\ne 0$, 
then $\eps_p(\te_p{\bf M})=\eps_p({\bf M})-1$ and 
$\eps_q(\te_p{\bf M})=\eps_q({\bf M})$.
\vskip 1mm
\noindent
{\rm (3)} For $q,q'\in\nz$ with $|q-q'|\geq 2$, we have 
$\te_q\te_{q'}=\te_{q'}\te_q$, $\tf_q\tf_{q'}=\tf_{q'}\tf_q$ and 
$\te_q\tf_{q'}=\tf_{q'}\te_q$
as operators on $\cBZ_{\nz}^e\cup\{0\}$.
\vskip 1mm
\noindent
{\rm (4)} For ${\bf M}\in \cBZ_{\nz}^e$, we have $\eps_p(\sigma({\bf M}))=
\eps_{\sigma^{-1}(p)}({\bf M})$.
\vskip 1mm
\noindent
{\rm (5)} The equalities $\sigma\circ\te_p=\te_{\sigma(p)}\circ\sigma$ and 
$\sigma\circ\tf_p=\tf_{\sigma(p)}\circ\sigma$ hold on $\cBZ_{\nz}^e\cup\{0\}$.
\end{lemma}
Next, let us define the ordinary $U_q(\widehat{\gtsl}_l)$-crystal structure on 
$(\cBZ_{\nz}^e)^{\sigma}$. Recall that the map $\mbox{wt}:(\cBZ_{\nz}^e)^{\sigma}
\to \widehat{P}$ is already defined. Here, $\widehat{P}$ is the weight 
lattice for $\widehat{\gtsl}_{l}$. For ${\bf M}\in (\cBZ_{\nz}^e)^{\sigma}$ and
$p\in \widehat{I}$, we define
$$\heps_p({\bf M}):=\eps_p({\bf M}),\quad
\hvphi_p({\bf M}):=\heps_p({\bf M})
+\langle \widehat{h}_p,\mbox{wt}({\bf M})\rangle.$$
For given ${\bf k}\in\cM_{\nz}$ and $p\in\widehat{I}$, we set
$L^e({\bf k},p):=\{q\in p+l\nz~|~\langle h_q,{\bf k}\rangle_{\nz}\ne 0 \}$;
note that $L^e({\bf k},p)$ is a finite set.
For ${\bf M}\in (\cBZ_{\nz}^e)^{\sigma}$, we define
$$\hte_p{\bf M}:=\begin{cases}
{\bf M}^{(1)} & \mbox{if }~\heps_p({\bf M})>0,\\
0 & \mbox{if }~\heps_p({\bf M})=0,
\end{cases}\quad\mbox{and}\quad
\htf_p{\bf M}={\bf M}^{(2)},$$
where ${\bf M}^{(i)}=(M^{(i)}_{\bf k})_{{\bf k}\in\cM_{\nz}}$, $i=1,2$, are the 
collections of integers defined by
$$M_{\bf k}^{(1)}:=\bigl(e_{L^e({\bf k},p)}{\bf M}\bigr)_{\bf k},\quad
M_{\bf k}^{(2)}:=\bigl(f_{L^e({\bf k},p)}{\bf M}\bigr)_{\bf k}\quad\mbox{for each }
{\bf k}\in\cM_{\nz}.$$

\begin{prop}\label{prop:well-def-aff} 
Let ${\bf M}\in (\cBZ_{\nz}^e)^{\sigma}$ and $p\in\widehat{I}$.
Then, we have $\hte_p{\bf M}\in (\cBZ_{\nz}^e)^{\sigma}\cup\{0\}$ and 
$\htf_p{\bf M}\in (\cBZ_{\nz}^e)^{\sigma}$.
\end{prop}
In order to prove the proposition above, we need the next lemma.
\begin{lemma}\label{lemma:well-def-aff}
For given ${\bf M}\in \cBZ_{\nz}^e$ and ${\bf k}\in\cM_{\nz}$, 
there exists a finite interval $I=[n_I+1,n_I+m_I]$ such that
for every $J=[n_J+1,n_J+m_J]$ with $n_J<n_I$ and $n_J+m_J>n_I+m_I$, 
\vskip 1mm
\noindent
{\rm (1)} $\bigl(\te_{n_J}{\bf M}\bigr)_{\Omega_J({\bf k})}=
\bigl(\te_{n_J+m_J+1}{\bf M}\bigr)_{\Omega_J({\bf k})}=
\bigl(\te_{n_J}\te_{n_J+m_J+1}{\bf M}\bigr)_{\Omega_J({\bf k})}
=M_{\Omega_J({\bf k})}$;
\vskip 1mm
\noindent
{\rm (2)} $\bigl(\tf_{n_J}{\bf M}\bigr)_{\Omega_J({\bf k})}=
\bigl(\tf_{n_J+m_J+1}{\bf M}\bigr)_{\Omega_J({\bf k})}=
\bigl(\tf_{n_J}\tf_{n_J+m_J+1}{\bf M}\bigr)_{\Omega_J({\bf k})}
=M_{\Omega_J({\bf k})}$.
\vskip 1mm
\noindent
Here we remark that the equalities in part {\rm (1)} hold under the assumption 
that $\te_{n_J}{\bf M}\ne \{0\}$ and $\te_{n_J+m_J+1}{\bf M}\ne \{0\}$.
\end{lemma}
\begin{proof}
We only prove that 
$$\bigl(\te_{n_J}\te_{n_J+m_J+1}{\bf M}\bigr)_{\Omega_J({\bf k})}
=M_{\Omega_J({\bf k})}\eqno{(4.3.1)}$$ 
under the condition that $\te_{n_J}{\bf M}\ne \{0\}$ and 
$\te_{n_J+m_J+1}{\bf M}\ne \{0\}$; the other equalities can be proved 
by a similar (and easier) argument.  

For the ${\bf M}$ and ${\bf k}$ above, take finite intervals 
$K=[n_K+1,n_K+m_K]$ and 
$I=[n_I+1,n_I+m_J]$ as in (4.2.2). Let $J=[n_J+1,n_J+m_J]\supsetneq I$, with 
$n_J<n_I$ and $n_J+m_J>n_I+m_I$. Take another finite interval
$L=[n_L+1,n_L+m_L]\supset J$ such that 
$$M_{\Omega_J({\bf k})}=
\bigl({\bf M}_L\bigr)_{{\rm res}_L(\Omega_J({\bf k}))},\quad
\bigl(\te_{n_J}\te_{n_J+m_J+1}{\bf M}\bigr)_{\Omega_J({\bf k})}=
\bigl(\te_{n_J}\te_{n_J+m_J+1}{\bf M}_L\bigr)_{{\rm res}_L(\Omega_J({\bf k}))},
$$
$$\eps_{n_J}({\bf M})=\eps_{n_J}({\bf M}_L),\quad\mbox{and}\quad
\eps_{n_J+m_J+1}({\bf M})=\eps_{n_J+m_J+1}({\bf M}_L).
$$
Note that such an interval $L$ always exists. Hence equation (4.3.1) 
is equivalent to
$$
\bigl(\te_{n_J}\te_{n_J+m_J+1}{\bf M}_L\bigr)_{{\rm res}_L(\Omega_J({\bf k}))}
=\bigl({\bf M}_L\bigr)_{{\rm res}_L(\Omega_J({\bf k}))}.
\eqno{(4.3.2)}$$
In what follows, we use the notation of Subsection 4.2. Namely, set
${\bf b}=(b_{k,l})_{(k,l)\in\Delta_L^+}
:=\Psi_L^{-1}({\bf M}_L)\in\cB_L$,
and
$\Lambda_{{\bf b}}:=\Xi_L^{-1}({\bf M}_L)\in\mbox{Irr}\Lambda(\nu_L)$, 
where
$\nu_L=\mbox{wt}({\bf M}_L)$.
Let 
$B^{L}=(B_{\tau}^{L})\in \Lambda_{{\bf b}}$ be a general point.

Since $L\supsetneq I$, we can show the following claim by an argument
similar to the one for Lemma \ref{lemma:claim2}:
\vskip 3mm
\noindent
{\bf Claim 2}. {\it Both of the composite maps
$$B^L_{\sigma(n_I\to n_K)}:V(\nu_J)_{n_I}\to V(\nu_J)_{n_K}\quad\mbox{and}$$
$$B^L_{\sigma(n_I+m_J+1\to n_K+m_K+1)}:V(\nu_J)_{n_I+m_I+1}\to 
V(\nu_J)_{n_K+m_K+1}$$
are zero maps.}
\vskip 3mm
Write
${\rm res}_L(\Omega_J({\bf k}))$ as a disjoint union of
finite intervals:
$${\rm res}_L(\Omega_J({\bf k}))=[s_1+1,t_1]\sqcup [s_2+1,t_2]\sqcup\cdots
\sqcup [s_l+1,t_l].$$
Then, by the construction, 
$$s_1=n_L,\quad t_1=n_J,\quad s_2=\mbox{min}\{q\in\nz~|~q\not\in {\bf k}\}-1
,$$
$$s_l=\mbox{max}\{q\in\nz~|~q\in {\bf k}\},\quad 
t_l=n_J+m_J+1,$$
and
$$s_1+1=n_L+1<t_1=n_J<n_I<n_K<s_2,\eqno{(4.3.3)}$$
$$s_l<n_K+m_K+1<n_I+m_I+1<n_J+m_J+1=t_l<n_L+m_L+1.\eqno{(4.3.4)}$$
From these, we deduce that
$$\out\bigl({\rm res}_L(\Omega_J({\bf k}))\bigr)=\{t_1,t_2,\cdots,t_l\},
\quad\inn\bigl({\rm res}_L(\Omega_J({\bf k}))\bigr)=\{s_2,s_3,\cdots,s_l\}.$$
Because
$$B_{\sigma(t_1\to s_2)}^L=B_{\sigma(n_K\to s_2)}^L\circ B_{\sigma(n_I\to n_K)}^L
\circ B_{\sigma(t_1\to n_I)}^L=0,$$
$$B_{\sigma(t_l\to s_l)}^L=B_{\sigma(n_K+m_K+1\to s_l)}^L\circ 
B_{\sigma(n_I+m_I+1\to n_K+m_K+1)}^L\circ B_{\sigma(t_l\to n_I+m_I+1)}^L=0$$
by Claim 2, (4.3.3), and (4.3.4), we see that
$$\begin{array}{ll}
\bigl({\bf M}_L\bigr)_{{\rm res}_L(\Omega_J({\bf k}))}
&=-\displaystyle{\dim_{\nc}\mbox{Coker}
\left(\mathop{\bigoplus}_{1\leq u\leq l}V(\nu_L)_{t_u}
\overset{\oplus B^L_{\sigma}}{\longrightarrow}
\mathop{\bigoplus}_{2\leq v\leq l}V(\nu_L)_{s_v}\right)}\\
&=-\displaystyle{\dim_{\nc}\mbox{Coker}
\left(\mathop{\bigoplus}_{2\leq u\leq l-1}V(\nu_L)_{t_u}
\overset{\oplus B^L_{\sigma}}{\longrightarrow}
\mathop{\bigoplus}_{2\leq v\leq l}V(\nu_L)_{s_v}\right)}.
\end{array}\eqno{(4.3.5)}$$
Let ${\bf b}_1:=\te_{n_J}\te_{n_J+m_J+1}{\bf b}$, and  
consider the corresponding irreducible Lagrangian 
$\Lambda_{{\bf b}_1}\in \mbox{Irr}\Lambda(\nu_L^1)$. 
Here we write $\nu_L^1=\nu_L-\alpha_{n_J}-\alpha_{n_J+m_L+1}$.
Let $B^{L}_1=\left((B^{L}_1)_{\tau}\right)\in 
\Lambda_{{\bf b}_1}$ be a general point. 
By the definitions of $\te_{n_J}$ and $\te_{n_J+m_J+1}$
(see Subsection 2.5), we may assume that
$$(B_1^L)_{\tau}=B_{\tau}^L\quad\mbox{if }\out(\tau)\ne n_J,n_J+m_J+1
\mbox{ and if }
\inn(\tau)\ne n_J,n_J,n_J+m_J+1.\eqno{(4.3.6)}$$
Then, by Claim 2, 
$$(B^L_1)_{\sigma(n_I\to n_K)}=0,\quad 
(B^L_1)_{\sigma(n_I+m_J+1\to n_K+m_K+1)}=0.$$
Therefore, by the same argument as for 
$\bigl({\bf M}_L\bigr)_{{\rm res}_L(\Omega_J({\bf k}))}$, we deduce that
$$\begin{array}{ll}
&\bigl(\te_{n_J}\te_{n_J+m_J+1}{\bf M}_L\bigr)_{{\rm res}_L(\Omega_J({\bf k}))}\\
&\qquad\qquad\qquad=
-\displaystyle{\dim_{\nc}\mbox{Coker}
\left(\mathop{\bigoplus}_{2\leq u\leq l-1}V(\nu_L^1)_{t_u}
\overset{\oplus (B^L_1)_{\sigma}}{\longrightarrow}
\mathop{\bigoplus}_{2\leq v\leq l}V(\nu_L^1)_{s_v}\right)}.
\end{array}
\eqno{(4.3.7)}$$
Since $n_J<s_2<t_2<\cdots<s_{l-1}<t_{l-1}<s_l<n_J+m_J+1$, 
the right-hand side of the last equality in (4.3.5) is equal to that of (4.3.7). 
Thus, we have proved equation (4.3.2). This completes the proof of the lemma.
\end{proof}
\vskip 3mm
\noindent
{\it Proof of Proposition \ref{prop:well-def-aff}}. 
We only prove that $\hte_p{\bf M}\in (\cBZ_{\nz}^e)^{\sigma}\cup\{0\}$. If 
$\heps_p({\bf M})=0$, then the assertion is obvious. 
So, we assume that $\heps_p({\bf M})>0$.\\

First, we prove that $\hte_p{\bf M}\in \cBZ_{\nz}^e$. Condition (2-a)
in Definition \ref{defn:BZ} can be checked by the same argument as in 
\cite{NSS}. Let us show that condition (2-b) is satisfied. 
Fix ${\bf k}\in \cM_{\nz}$ and take a finite interval $I=[n_I+1,n_I+m_I]$
satisfying condition (4.2.2), with ${\bf M}$ replaced by 
$\te_{L^e({\bf k},p)}{\bf M}$. 
Let $I'=[n_{I'}+1,n_{I'}+m_{I'}]$ be an interval such that
$n_{I'}<n_I$, $n_{I'}+m_{I'}>n_I+m_I$, and 
$I'\in \mbox{Int}^e(\te_{L^e({\bf k},p)}{\bf M},{\bf k})$. 
We will show that this $I'$ satisfies condition (2-b).

Take $J=[n_J+1,n_J+m_J]\supset I'$. By the definitions, we have
$$L^e(\Omega_{J}({\bf k}),p)=L^e({\bf k},p)\cup \delta(J,p;l).$$
Here,
$$\delta(J,p;l):=
\begin{cases}
\{n_J,n_J+m_J+1\} & 
\mbox{if }n_{J}\equiv p,~n_{J}+m_{J}+1\equiv p~(\mbox{mod } l),\\
\{n_{J}\} & 
\mbox{if }n_{J}\equiv p,~n_{J}+m_{J}+1\not\equiv p~(\mbox{mod } l),\\
\{n_{J}+m_{J}+1\} & 
\mbox{if }n_{J}\not\equiv p,~n_{J}+m_{J}+1\equiv p~(\mbox{mod } l),\\
\phi& 
\mbox{if }n_{J}\not\equiv p,~n_{I'}+m_{J}+1\not\equiv p~(\mbox{mod } l).
\end{cases}
$$
From this, we deduce by Lemma \ref{lemma:well-def-aff} that
$$(\hte_p{\bf M})_{\Omega_{J}({\bf k})}=
\bigl(\te_{L^e(\Omega_{J}({\bf k}),p)}{\bf M}\bigr)_{\Omega_{J}({\bf k})}
=\bigl(\te_{L^e({\bf k},p)}{\bf M}\bigr)_{\Omega_{I'}({\bf k})}\quad
\mbox{for every }J\supset I'.$$
Note that by Lemma \ref{lemma:propaties}, 
$\eps_q({\bf M})=\heps_p({\bf M})>0$ for every $q\in p+l\nz$.
Since $I'\in \mbox{Int}^e(\te_{L^e({\bf k},p)}{\bf M},{\bf k})$, we conclude that
$$\bigl(\te_{L^e({\bf k},p)}{\bf M}\bigr)_{\Omega_{J}({\bf k})}=
\bigl(\te_{L^e({\bf k},p)}{\bf M}\bigr)_{\Omega_{I'}({\bf k})}.
$$
This shows that condition (2-b) is satisfied.\\

It remains to show that $\hte_p{\bf M}$ is $\sigma$-invariant. However, this
follows easily from Lemma \ref{lemma:propaties}. This proves the proposition. 
\hfill$\square$
\vskip 3mm
Now we are ready to state one of the main results of this paper.
\begin{thm}\label{thm:ord-aff}
$\left((\cBZ_{\nz}^e)^{\sigma};\mbox{\rm wt},\heps_p,\hvphi_p,\hte_p,\htf_p
\right)$
is a $U_q(\widehat{\gtsl}_l)$-crystal.
\end{thm}

\begin{lemma}\label{lemma:wt-eps}
Let ${\bf M}\in (\cBZ_{\nz}^e)^{\sigma}$ and $p\in\widehat{I}$. Then, 
the following hold:
\vskip 1mm
\noindent
{\rm (1)} 
$\mbox{\rm wt}(\hte_p{\bf M})=\mbox{\rm wt}({\bf M})+\widehat{\alpha}_p$ if 
$\heps_p({\bf M})>0$, and 
$\mbox{\rm wt}(\htf_p{\bf M})=\mbox{\rm wt}({\bf M})-\widehat{\alpha}_p;$
\vskip 1mm
\noindent
{\rm (2)}
$\heps_p(\hte_p{\bf M})=\heps_p({\bf M})-1$ if $\heps_p({\bf M})>0$, and
$\heps_p(\htf_p{\bf M})=\heps_p({\bf M})+1.$
\end{lemma}

\begin{proof} We only prove the first equation of part (1), since
the other ones follow by a similar (and easier) argument.

Let ${\bf M}\in (\cBZ_{\nz}^e)^{\sigma}$, with $\heps_p({\bf M})>0$, and 
$J=[n_{J}+1,n_{J}+m_{J}]\in 
\bigcap_{q\in \widehat{I}}~\mbox{Int}^e(\hte_p{\bf M},{\bf k}(\Lambda_q))$. 
Then, 
\begin{align*}
\mbox{wt}(\hte_p{\bf M})&=\sum_{q\in\widehat{I}}
\Theta\bigl(\hte_p{\bf M}\bigr)_{{\bf k}(\Lambda_q)}\widehat{\alpha}_q\\
&=\sum_{q\in\widehat{I}}\bigl(\hte_p{\bf M}\bigr)_{\Omega_{J}({\bf k}(\Lambda_q))}
\widehat{\alpha}_q\\
&=\sum_{q\in\widehat{I}}
\bigl(\te_{L^e(\Omega_{J}({\bf k}(\Lambda_q)),p)}{\bf M}
\bigr)_{\Omega_{J}({\bf k}(\Lambda_q))}\widehat{\alpha}_q.
\end{align*}
Here we note that
$$L^e(\Omega_{J}({\bf k}(\Lambda_q)),p)=L^e({\bf k}(\Lambda_q),p)\cup
\delta(J,p;l)=
\begin{cases}
\{q\}\cup \delta(J,p;l)& 
\mbox{if }p=q,\\
\delta(J,p;l)& 
\mbox{if }p\ne q.
\end{cases}
$$
Now we assume that $J$ is sufficiently large. More precisely, for each 
$q\in \widehat{I}$, let us take an interval $I'=I'_q$ as in the proof of 
Proposition \ref{prop:well-def-aff}, and then take $J$ in such a way that 
$J\supset \bigcup_{q\in \widehat{I}}~I'_q$. Then, by Lemma 
\ref{lemma:well-def-aff}, we deduce that
$$\bigl(\te_{L^e(\Omega_{J}({\bf k}(\Lambda_q)),p)}{\bf M}
\bigr)_{\Omega_{J}({\bf k}(\Lambda_q))}
=\begin{cases}
({\bf M})_{\Omega_{J}({\bf k}(\Lambda_p))}+1 & \mbox{if }p=q,\\ 
({\bf M})_{\Omega_{J}({\bf k}(\Lambda_p))} & \mbox{if }p\ne q.
\end{cases}
$$
Therefore, we obtain
$$\mbox{\rm wt}(\hte_p{\bf M})=\mbox{\rm wt}({\bf M})+\widehat{\alpha}_p.$$
\end{proof}
\vskip 5mm
\noindent
{\it Proof of Theorem \ref{thm:ord-aff}.}
By Lemma \ref{lemma:wt-eps}, it suffices to prove the following:
$$\hte_p\htf_p{\bf M}={\bf M}\quad
\mbox{for every }{\bf M}\in(\cBZ_{\nz}^e)^{\sigma}\mbox{ and }p\in\widehat{I}.$$
Since this follows easily from Lemmas \ref{lemma:propaties} and
\ref{lemma:well-def-aff}, we omit the details of its proof.
\hfill $\square$
\subsection{Uniqueness of an element of weight zero}
It is easy to show the following lemma.
\begin{lemma}\label{lemma:non-pos}
Let ${\bf M}=(M_{\bf k})_{{\bf k}\in\cM_{\nz}}\in\cBZ_{\nz}^e$. Then, 
each component $M_{\bf k}$ for ${\bf k}\in\cM_{\nz}$ is a nonpositive integer.
\end{lemma}
The next corollary is a direct consequence of this lemma.
\begin{cor}\label{cor:Q^-}
Let ${\bf M}\in (\cBZ_{\nz}^e)^{\sigma}$. Then, $\mbox{\rm wt}({\bf M})\in
\widehat{Q}^-$.
\end{cor}

\begin{prop}\label{prop:eps=0}
For ${\bf M}\in (\cBZ_{\nz}^e)^{\sigma}$, the following are equivalent.
\vskip 1mm
\noindent
{\rm (a)} $\heps_p({\bf M})=0$ for every $p\in \widehat{I}$.
\vskip 1mm
\noindent
{\rm (b)} ${\bf M}={\bf O}^*$.
\end{prop}
\begin{proof}
Since (b) $\Rightarrow$ (a) is obvious, we will prove that $M_{\bf k}=0$ 
for all ${\bf k}\in \cM_{\nz}$ under the assumption that
$\heps_p({\bf M})=0$ for every $p\in \widehat{I}$.

We note that $M_{{\bf k}(\Lambda_q)}=0$ for every $q\in\nz$ by 
the normalization condition. Let ${\bf k}\in\cM_{\nz}\setminus 
\left(\bigcup_{q\in\nz}~{\bf k}(\Lambda_q)\right)$. Then there exists
the smallest finite interval $I_{\bf k}$ such that 
${\bf k}\in \cM_{\nz}(I_{\bf k})$. 
We prove the assertion above by induction on $t:=|I_{\bf k}|\geq 1$.\\

Assume that $t=1$. Then, ${\bf k}=\sigma_q{\bf k}(\Lambda_q)=
\nz_{\leq q-1}\cup \{q+1\}$ for some $q\in \nz$. If we take $q'\in \widehat{I}$
such that $q\equiv q'~(\mbox{mod}~l)$, then we have
$$M_{\sigma_q{\bf k}(\Lambda_q)}=\eps_q({\bf M})=\heps_{q'}({\bf M})=0.$$

Now, we assume that $t>1$, and  
\begin{itemize}
\item[(i)] the assertion holds for every ${\bf m}\in \cM_{\nz}$ with 
$|I_{\bf m}|<t$. 
\end{itemize}
{\bf Step 1}. Let ${\bf k}=\nz_{\leq n}\cup\{n+t+1\}$ 
for some $n\in\nz$. We use the tropical Pl\"ucker relation for 
$i=n+1,j=n+t,k=n+t+1$:
\begin{align*}
&M_{\nz_{\leq n}\cup\{n+t\}}+M_{\nz_{\leq n}\cup\{n+1,n+t+1\}}\\
&\qquad=\mbox{min}
\left\{M_{\nz_{\leq n}\cup\{n+1\}}+M_{\nz_{\leq n}\cup\{n+t,n+t+1\}},~
M_{\nz_{\leq n}\cup\{n+t+1\}}+M_{\nz_{\leq n}\cup\{n+1,n+t\}}\right\}.
\end{align*}
By the assumption (i), we see that
$$M_{\nz_{\leq n}\cup\{n+t\}}=M_{\nz_{\leq n}\cup\{n+1,n+t+1\}}=
M_{\nz_{\leq n}\cup\{n+1\}}=M_{\nz_{\leq n}\cup\{n+1,n+t\}}=0.$$
Since $M_{\nz_{\leq n}\cup\{n+t+1\}}=M_{\bf k}$ and 
$M_{\nz_{\leq n}\cup\{n+t,n+t+1\}}=M_{{\bf k}\cup\{n+t\}}$ are both 
nonpositive integers, we obtain
$$M_{\bf k}=M_{{\bf k}\cup\{n+t\}}=0.$$
{\bf Step 2}. Let ${\bf k}=\nz_{\leq n}\cup\{k_1<\cdots<k_r\}$, with 
$k_1=n+s+1~(1<s\leq t)$, and $k_r=n+t+1$. We prove the assertion by  
descending induction on s. If $s=t$, then $r=1$ and the assertion is
already proved in Step 1. Assume that 
\begin{itemize}
\item[(ii)] the assertion holds for every 
${\bf m}=\nz_{\leq n}\cup\{m_1<\cdots<m_{r'}\}$, with $m_1=n+s'+1>k_1$, and
$m_{r'}=n+t+1$.
\end{itemize} 
Set ${\bf k}':={\bf k}\setminus\{n+s+1,n+t+1\}$, and use the 
tropical Pl\"ucker relation for $i=n+1,j=n+s+1,k=n+t+1$, and ${\bf k}'$:
\begin{align*}
&M_{{\bf k}'\cup\{n+s+1\}}+M_{{\bf k}'\cup\{n+1,n+t+1\}}\\
&\qquad=\mbox{min}
\left\{M_{{\bf k}'\cup\{n+1\}}+M_{{\bf k}'\cup\{n+s+1,n+t+1\}},~
M_{{\bf k}'\cup\{n+t+1\}}+M_{{\bf k}'\cup\{n+1,n+s+1\}}\right\}.
\end{align*}
By the assumption (i), we obtain
$$M_{{\bf k}'\cup\{n+s+1\}}=M_{{\bf k}'\cup\{n+1,n+t+1\}}=M_{{\bf k}'\cup\{n+1\}}
=M_{{\bf k}'\cup\{n+1,n+s+1\}}=0.$$
Also, we have 
$$M_{{\bf k}'\cup\{n+t+1\}}=0$$
by the assumption (ii). Therefore, by Lemma \ref{lemma:non-pos}, we conclude
that
$$M_{\bf k}=M_{{\bf k}'\cup\{n+s+1,n+t+1\}}=0.$$
This proves the proposition.
\end{proof}

The following corollary is a key to the proof of the connectedness of 
the crystal graph of the $U_q(\widehat{\gtsl}_l)$-crystal 
$\left((\cBZ_{\snz}^{e})^{\sigma};\mbox{\rm wt},\widehat{\eps}_p^*,
\widehat{\vphi}_p^*,\hte_p^*,\htf_p^*\right)$, which will be given in the next
section.
\begin{cor}
${\bf O}^*$ is the unique element of $(\cBZ_{\nz}^e)^{\sigma}$ of 
weight zero. 
\end{cor}
\begin{proof}
It suffices to show the following:
\begin{center}
{\it If ${\bf M}\ne {\bf O}^*$, then $\mbox{\rm wt}({\bf M})\ne 0$.}
\end{center}
Let ${\bf M}\ne {\bf O}^*$. By Proposition \ref{prop:eps=0}, there exists
$p\in\widehat{I}$ such that $\heps_p({\bf M})>0$. This implies that
$\hte_p{\bf M}\in (\cBZ_{\nz}^e)^{\sigma}$. Therefore, by Corollary \ref{cor:Q^-}, 
we have
$$\mbox{wt}(\hte_p{\bf M})\in \widehat{Q}^-.\eqno{(4.4.1)}$$
Also, because 
$\left((\cBZ_{\snz}^{e})^{\sigma};\mbox{\rm wt},\widehat{\eps}_p,
\widehat{\vphi}_p,\hte_p,\htf_p\right)$ is a $U_q(\widehat{\gtsl}_l)$-crystal
(Theorem \ref{thm:ord-aff}), we have
$$\mbox{wt}(\hte_p{\bf M})=\mbox{wt}({\bf M})+\widehat{\alpha}_p.
\eqno{(4.4.2)}$$
Now, suppose that $\mbox{wt}({\bf M})=0$. Then, by (4.4.2), we obtain
$$\mbox{wt}(\hte_p{\bf M})=\widehat{\alpha}_p,$$
which contradicts (4.4.1). Thus, we conclude that 
$\mbox{\rm wt}({\bf M})\ne 0$.  This proves the corollary.
\end{proof}
\subsection{Some other properties}
The results of this subsection will be used in the next section. 
\begin{lemma}\label{lemma:eps1}
Let $p,q\in\nz$ with $p\ne q$, and ${\bf M}\in \cBZ_{\nz}^e$.
\vskip 1mm
\noindent
{\rm (1)} If $\eps_p({\bf M})>0$, then 
$\eps_q^*(\te_p{\bf M})=\eps_q^*({\bf M})$.
\vskip 1mm
\noindent
{\rm (2)} If $\eps_p^*({\bf M})>0$, then 
$\eps_q(\te_p^*{\bf M})=\eps_q({\bf M})$.
\end{lemma}
\begin{proof}
Because part (2) can be proved by a similar (and easier) argument,
we only give a proof of part (1). By the definitions, we have
$$\begin{array}{lll}
\eps_q^*(\te_p{\bf M})
&=&-~\Theta\bigl((\te_p{\bf M})^*\bigr)_{{\bf k}(\Lambda_q)}
-\Theta\bigl((\te_p{\bf M})^*\bigr)_{{\bf k}(\sigma_q\Lambda_q)}\\
&&\qquad\qquad\qquad
+~\Theta\bigl((\te_p{\bf M})^*\bigr)_{{\bf k}(\Lambda_{q+1})}
+\Theta\bigl((\te_p{\bf M})^*\bigr)_{{\bf k}(\Lambda_{q-1})}.
\end{array}\eqno{(4.5.1)}$$ 
For simplicity of notation, we write
${\bf k}_1={\bf k}(\Lambda_q),{\bf k}_2={\bf k}(\sigma_q\Lambda_q),
{\bf k}_3={\bf k}(\Lambda_{q+1}),{\bf k}_4={\bf k}(\Lambda_{q-1})$.
Take a finite interval $I$ such that $p\in I$ and 
$I\in \bigcap_{k=1}^4~
\mbox{Int}^c\bigl((\te_p{\bf M})^*,{\bf k}_k\bigr)$.  

Let us compute the second term on the right-hand side of (4.5.1). 
\begin{align*}
\Theta\bigl((\te_p{\bf M})^*\bigr)_{{\bf k}(\sigma_q\Lambda_q)}
=\bigl((\te_p{\bf M})^*\bigr)_{\Omega_{I}^c({\bf k}(\sigma_q\Lambda_q)^c)}
=(\te_p{\bf M})_{\Omega_{I}({\bf k}(\sigma_q\Lambda_q))}.
\end{align*}
Since $p\ne q$, the following two cases occur:
\vskip 1mm
case (a): $p=q\pm 1$ ($\Leftrightarrow~
\Omega_{I}({\bf k}(\sigma_q\Lambda_q))\in \cM_{\nz}(p)^*$), 
\vskip 1mm
case (b): $|p-q|\geq 2$
$\left(\Leftrightarrow~\Omega_{I}({\bf k}(\sigma_q\Lambda_q))\in 
\cM_{\nz}\setminus \bigl(\cM_{\nz}(p)\cup\cM_{\nz}(p)^*\bigr)\right)$.
\vskip 3mm
\noindent
By Lemma \ref{cor:ord-kas}, we have
$$(\te_p{\bf M})_{\Omega_{I}({\bf k}(\sigma_q\Lambda_q))}
=\begin{cases}
M_{\Omega_{I}({\bf k}(\sigma_q\Lambda_q))}+1 & \mbox{in case (a)},\\
M_{\Omega_{I}({\bf k}(\sigma_q\Lambda_q))} & \mbox{in case (b)}.
\end{cases}
$$

\vskip 3mm
By a similar computation, we obtain
$$\Theta\bigl((\te_p{\bf M})^*\bigr)_{{\bf k}(\Lambda_q)}
=M_{\Omega_{I}({\bf k}(\Lambda_{q}))},$$
$$(\te_p{\bf M})_{\Omega_{I}({\bf k}(\Lambda_{q+1}))}
=\begin{cases}
M_{\Omega_{I}({\bf k}(\Lambda_{q+1}))}+1 & \mbox{if }p=q+1,\\
M_{\Omega_{I}({\bf k}(\Lambda_{q+1}))} & \mbox{otherwise},
\end{cases}$$
$$
(\te_p{\bf M})_{\Omega_{I}({\bf k}(\Lambda_{q-1}))}
=\begin{cases}
M_{\Omega_{I}({\bf k}(\Lambda_{q-1}))}+1 & \mbox{if }p=q-1,\\
M_{\Omega_{I}({\bf k}(\Lambda_{q-1}))} & \mbox{otherwise}.
\end{cases}
$$
Combining the above, we deduce that
\begin{align*}
\eps_q^*(\te_p{\bf M})
&=-M_{\Omega_{I}({\bf k}(\Lambda_{q}))}
-M_{\Omega_{I}({\bf k}(\sigma_q\Lambda_q))} 
+M_{\Omega_{I}({\bf k}(\Lambda_{q+1}))}
+M_{\Omega_{I}({\bf k}(\Lambda_{q-1}))}\\
&=\eps_q^*({\bf M}).
\end{align*} 
This proves the lemma.
\end{proof}

\begin{prop}\label{prop:aff-strict}
Let $p,q\in \widehat{I}$, and ${\bf M}\in (\cBZ_{\nz}^e)^{\sigma}$.
Set $c:=\heps_p({\bf M})$ and ${\bf M}':=\hte_p^c{\bf M}$. 
\vskip 1mm
\noindent
{\rm (1)} We have
$$\heps^*_p({\bf M})=\mbox{\rm max}\left\{\heps^*_p({\bf M}'),~c-
\bigl\langle\widehat{h}_p,\mbox{\rm wt}({\bf M}')\bigr\rangle\right\}.$$
{\rm (2)} If $p\ne q$ and $\heps_q^*({\bf M})>0$, then
$$\heps_q(\hte_p^*{\bf M})=c,\qquad \hte_q^c(\hte_p^*{\bf M})=\hte_p{\bf M}'.$$
{\rm (3)} If $\heps_p^*({\bf M})>0$, then
$$\heps_p(\hte_p^*{\bf M})=\begin{cases}
\heps_p({\bf M}) & \mbox{if}\quad\heps^*_p({\bf M}')\geq 
c-\bigl\langle\widehat{h}_p,\mbox{\rm wt}({\bf M}')\bigr\rangle,\\
\heps_p({\bf M})-1& \mbox{if}\quad\heps^*_p({\bf M}')<
c-\bigl\langle\widehat{h}_p,\mbox{\rm wt}({\bf M}')\bigr\rangle,
\end{cases}$$
and
$$\hte_p^{c'}\bigl(\hte_p^*{\bf M}\bigr)=
\begin{cases}
\hte_p^*{\bf M}' & \mbox{if}\quad\heps^*_p({\bf M}')\geq 
c-\bigl\langle\widehat{h}_p,\mbox{\rm wt}({\bf M}')\bigr\rangle,\\
{\bf M}' & \mbox{if}\quad\heps^*_p({\bf M}')<
c-\bigl\langle\widehat{h}_p,\mbox{\rm wt}({\bf M}')\bigr\rangle.
\end{cases}$$
Here, we set $c':=\heps_p(\hte_p^*{\bf M}).$
\end{prop}

\begin{proof}
By taking a sufficiently large finite interval $I$, each of the equations above 
follows from the corresponding one in the case of finite intervals (Proposition 
\ref{prop:fin-strict}). 

As an example, let us show part (1). By the definitions and Lemma 
\ref{lemma:eps1}, it suffices to show that
$$\eps^*_p({\bf M})=\mbox{\rm max}\left\{\eps^*_p(\te_p^c{\bf M}),~-c-
\bigl\langle\widehat{h}_p,\mbox{\rm wt}({\bf M})\bigr\rangle\right\}.$$
Let ${\bf k}_k$, $k=1,2,3,4$, be the Maya diagrams which we introduced in the
proof of Lemma \ref{lemma:eps1}. Note that there exists a finite interval $I$ 
such that
\vskip 1mm
(a) $I\in\left(\displaystyle{\bigcap_{k=1}^4}\mbox{Int}^c({\bf M}^*,{\bf k}_k)\right)
\cap \left(\displaystyle{\bigcap_{k=1}^4}\mbox{Int}^c((\te_p^c{\bf M})^*,{\bf k}_k)\right)$,
\vskip 1mm
(b) $(\te_p^c{\bf M})_I=\te_p^c{\bf M}_I$,
\vskip 1mm
(c) $c=\eps_p({\bf M})=\eps_p({\bf M}_I)$,
\vskip 1mm
(d) $\displaystyle{\bigl\langle\widehat{h}_p,\mbox{\rm wt}({\bf M})\bigr\rangle}
=\langle h_p,\mbox{wt}({\bf M}_I)\rangle_I.$
\vskip 3mm
\noindent
For such an interval $I$, we have 
\begin{align*}
\eps^*_p(\te_p^c{\bf M}) &= \eps_p\bigl((\te_p^c{\bf M})^*\bigr)\\
&=\eps_p\left(\bigl((\te_p^c{\bf M})^*\bigr)_I\right)\quad\mbox{by (a)}\\
&=\eps_p\left(\bigl((\te_p^c{\bf M})_I\bigr)^*\right)\\
&=\eps_p\left(\bigl(\te_p^c{\bf M}_I\bigr)^*\right)\quad\mbox{by (b)}\\
&=\eps_p^*\left(\te_p^{\eps_p({\bf M}_I)}{\bf M}_I\right)\quad\mbox{by (c)}.
\end{align*}
Similarly, we obtain $\eps^*_p({\bf M})=\eps^*_p({\bf M}_I)$. 
Therefore, it suffices to show that
$$\eps^*_p({\bf M}_I)
=\mbox{\rm max}\left\{\eps^*_p({\bf M}'_I),~\eps_p({\bf M}_I)
-\bigl\langle {h}_p,\mbox{\rm wt}({\bf M}_I')\bigr\rangle_I\right\},
\eqno{(4.5.2)}$$
where we set ${\bf M}_I':=\te_p^{\eps_p({\bf M}_I)}{\bf M}_I$. Here, we note that
equation (4.5.2) 
is just the equation in part (1) of Proposition \ref{prop:fin-strict}. Thus, we have
shown part (1).

Since the other equations are shown in a similar way, we omit 
the details of their proofs. 
\end{proof}
\section{Proof of the connectedness of  
$\left((\cBZ_{\snz}^{e})^{\sigma};\mbox{\rm wt},\widehat{\eps}_p^*,
\widehat{\vphi}_p^*,\hte_p^*,\htf_p^*\right)$}
\subsection{Strategy}
The aim of this section is to prove the following theorem.
\begin{thm}[Main theorem]\label{thm:main}
As a crystal, $\left((\cBZ_{\snz}^{e})^{\sigma};\mbox{\rm wt},\widehat{\eps}_p^*,
\widehat{\vphi}_p^*,\hte_p^*,\htf_p^*\right)$ is isomorphic to $B(\infty)$
for $U_q(\widehat{\gtsl}_l)$. In particular, the crystal graph of this crystal 
is connected.
\end{thm}
In order to prove this theorem, we use a characterization of 
$B(\infty)$, which was obtained in \cite{KS}; although it is valid 
for an arbitrary symmetrizable Kac-Moody Lie algebra, we restrict 
ourselves to the case of type $A_{l-1}^{(1)}$.

For $p\in \widehat{I}$, we define a crystal 
$\left(B_p^*;\mbox{\rm wt},{\heps}_p^*,{\hvphi}_p^*,\hte_p^*,\htf_p^*\right)$ 
as follows. 
$$B_p^*:=\{b_p^*(n)~|~n\in\nz\},$$
$$\mbox{wt}(b_p^*(n)):=n\widehat{\alpha}_p,\quad 
\heps_q^*(b_p^*(n)):=\begin{cases}
-n & \mbox{if }~q=p,\\
-\infty & \mbox{if }~q\ne p,
\end{cases}\quad
\hvphi_q^*(b_p^*(n)):=\begin{cases}
n & \mbox{if }~q=p,\\
-\infty & \mbox{if }~q\ne p,
\end{cases}
$$
$$\hte_q^*(b_p^*(n)):=\begin{cases}
b_p^*(n+1) & \mbox{if }~q=p,\\
0 & \mbox{if }~q\ne p,
\end{cases}\qquad
\htf_q^*(b_p^*(n)):=\begin{cases}
b_p^*(n-1) & \mbox{if }~q=p,\\
0 & \mbox{if }~q\ne p.
\end{cases}$$
For simplicity of notation, we set $b_p^*:=b_p^*(0)$.

\begin{thm}[\cite{KS}]\label{thm:seven-cond}
Let $\left(B;\mbox{\rm wt},{\heps}_p^*,{\hvphi}_p^*,\hte_p^*,\htf_p^*\right)$ be 
a $U_q(\widehat{\gtsl}_l)$-crystal,
and let $b_{\infty}^*$ be an element of $B$ of weight zero. We assume that
the following seven conditions are satisfied.
\begin{itemize}
\item[(1)] $\mbox{\rm wt}(B)\subset \widehat{Q}^-$.
\item[(2)] $b_{\infty}^*$ is the unique element of $B$ of weight zero.
\item[(3)] $\heps_p^*(b_{\infty}^*)=0$ for all $p\in\widehat{I}$.
\item[(4)] $\heps_p^*(b)\in\nz$ for all $p\in\widehat{I}$ and $b\in B$.
\item[(5)] For each $p\in\widehat{I}$, there exists a strict embedding 
$\Psi_p^*:B\to B\otimes B_p^*$.
\item[(6)] $\Psi_p^*(B)\subset B\otimes \left\{\left.(\htf_p^*)^nb_p^*~\right|
~n\geq 0\right\}$ for all $p\in\widehat{I}$.
\item[(7)] For every $b\in B$ such that $b\ne b_{\infty}^*$, there exists 
$p\in \widehat{I}$ such that
$\Psi_p^*(b)=b'\otimes (\htf_p^*)^nb_p^*$ with $n>0$.
\end{itemize}
Then, $\left(B;\mbox{\rm wt},{\heps}_p^*,{\hvphi}_p^*,\hte_p^*,\htf_p^*\right)$ 
is isomorphic as a crystal to $B(\infty)$.
\end{thm}

Let us check the seven conditions above for 
the crystal $\left((\cBZ_{\snz}^{e})^{\sigma};\mbox{\rm wt},\widehat{\eps}_p^*,
\widehat{\vphi}_p^*,\hte_p^*,\htf_p^*\right)$, with $b_{\infty}^*={\bf O}^*$. 
Conditions (1)$\sim$(4) are obvious from the definitions. In the next
subsection, we construct a strict embedding 
$\Psi_p^*:(\cBZ_{\snz}^{e})^{\sigma}\to (\cBZ_{\snz}^{e})^{\sigma}\otimes B_p^*$
for each $p\in\widehat{I}$, and check conditions (6) and (7).
\begin{rem}
Since our aim is to prove Theorem \ref{thm:main}, we consider 
the $\ast$-crystal structure on $(\cBZ_{\snz}^{e})^{\sigma}$, not the
ordinary crystal structure. 
\end{rem}
\subsection{Proof of Theorem \ref{thm:main} and the connectedness}

\begin{defn}
Let $p\in\widehat{I}$. We define a map 
$\Psi_p^*: (\cBZ_{\snz}^{e})^{\sigma}\to (\cBZ_{\snz}^{e})^{\sigma}\otimes B_p^*$ by
$\Psi_p({\bf M}):={\bf M}'\otimes (\htf_p^*)^cb_p^*$. Here, $c:=\heps_p({\bf M})$
and ${\bf M}':=\hte_p^c{\bf M}$.
\end{defn}
The following lemma is obvious from the definitions.
\begin{lemma}\label{lemma:inj-wt}
{\rm (1)} $\Psi_p$ is an injective map.
\vskip 1mm
\noindent
{\rm (2)} For every ${\bf M}\in (\cBZ_{\snz}^{e})^{\sigma}$, we have
$\mbox{\rm wt}(\Psi_p^*({\bf M}))=\mbox{\rm wt}({\bf M})$. 
\end{lemma}
\vskip 3mm
\noindent
{\it Proof of Theorem \ref{thm:main}}.
If condition (5) is satisfied for the $\Psi_p^*$ above, then conditions (6) and (7)
are automatically satisfied by the definitions. Therefore, the remaining task 
is to check condition (5). However, by an argument similar to the one in \cite{KS}, 
this follows from Proposition \ref{prop:aff-strict}. Thus, we
have established the theorem.\hfill$\square$

\begin{cor}\label{cor:conn} 
{\rm (1)} $(\cBZ_{\nz}^e)^{\sigma}({\bf O}^*)=(\cBZ_{\nz}^e)^{\sigma}.$
\vskip 1mm
\noindent
{\rm (2)} $\cBZ_{\nz}^{\sigma}({\bf O})=\cBZ_{\nz}^{\sigma}.$
\end{cor}
\begin{proof}
(1) is a direct consequence of the main theorem. Applying the map $\ast$ on
both sides of (1), we obtain (2).
\end{proof}
The second equality above is what we announced in the ``note added in proof''
of \cite{NSS}.


\begin{thebibliography}{[MFT50]}
\bibitem[BKT]{BKT}
P. Baumann, J. Kamnitzer, and P. Tingley,
{\it Affine Mirkovi\'c-Vilonen polytopes},
arXiv:1110.3661. 
\bibitem[BFZ]{BFZ}
A. Berenstein, S. Fomin, and A. Zelevinsky, 
{\it Parametrizations of canonical bases and totally positive matrices}, 
Adv. Math. {\bf 122} (1996), 49-149. 
\bibitem[BFG]{BFG}
A. Braverman, M. Finkelberg, and D. Gaitsgory
{\it Uhlenbeck spaces via affine Lie algebras}, 
In The Unity of Mathematics (volume dedicated to I. M. Gelfand 
in honor of his 90th birthday), Progr. Math., {\bf 244} (2006), 17-135, 
Birkh\"auser.
\bibitem[K1]{Kam1}
J. Kamnitzer,
{\it Mirkovi\'{c}-Vilonen cycles and polytopes},
Ann. of Math. {\bf 171} (2010), 245-294.
\bibitem[K2]{Kam2}
J. Kamnitzer,
{\it The crystal structure on the set of Mirkovi\'{c}-Vilonen polytopes},
Adv. Math. {\bf 215} (2007), 66-93. 
\bibitem[KS]{KS}
M. Kashiwara and Y. Saito,
{\it Geometric construction of crystal bases},
Duke Math. J. {\bf 89} (1997), 9-36.
\bibitem[M]{M}
D. Muthiah,
{\it Double MV cycles and the Naito-Sagaki-Saito crystal},
arXiv:1108.5404.
\bibitem[NSS1]{NSS}
S. Naito, D. Sagaki, and Y. Saito, 
{\it Toward Berenstein-Zelevinsky data in affine type $A$, I: 
Construction of affine analogs}, 
Contemp. Math. {\bf 565} (2012), 143-184.
\bibitem[NSS2]{NSS2}
S. Naito, D. Sagaki, and Y. Saito, 
{\it Toward Berenstein-Zelevinsky data in affine type $A$, II: 
Explicit description}, 
Contemp. Math. {\bf 565} (2012), 185-216.
\bibitem[S]{S}
Y. Saito, 
{\it Mirkovi\'c-Vilonen polytopes and a quiver construction of
crystal basis in type $A$}, 
arXiv:1010.0086, to appear in IMRN.
\end{thebibliography}
\end{document}